\documentclass[12pt]{amsart}
\usepackage{amssymb}

\usepackage{amsmath,amsfonts,euscript,amscd,amsthm,amssymb,upref,graphics}
\theoremstyle{plain}
\newtheorem{theorem}{Theorem}
\swapnumbers

\newtheorem{proposition}[subsection]{Proposition}
\newtheorem{lemma}[subsection]{Lemma}
\newtheorem{corollary}[subsection]{Corollary}

\newtheorem{convention}[subsection]{Convention}
\newtheorem{notation}[subsection]{Notation}

\theoremstyle{definition}
\newtheorem{definition}[subsection]{Definition}
\newtheorem{example}[subsection]{Example}
\newtheorem{remark}[subsection]{Remark}

\newtheorem{nothing*}[subsection]{}
\newcommand{\rien}[1]{}

\newcommand{\diver}{ \operatorname{{\rm div}}}
\newcommand{\Lie}{ \operatorname{{\rm Lie}}}

\newcommand{\VFH}{ \operatorname{{\rm VF}_{hol}}}

\newcommand{\LieH}{ \operatorname{{\rm Lie}_{hol}}}
\newcommand{\VFA}{ \operatorname{{\rm VF}_{alg}}}
\newcommand{\LieA}{ \operatorname{{\rm Lie}_{alg}}}

\newcommand{\VFHO}{ \operatorname{{\rm VF}_{hol}^\omega}}
\newcommand{\LieHO}{ \operatorname{{\rm Lie}_{hol}^\omega}}
\newcommand{\VFAO}{ \operatorname{{\rm VF}_{alg}^\omega}}
\newcommand{ \LieAO}{ \operatorname{{\rm Lie}_{alg}^\omega}}
\newcommand{\IVFHO}{ \operatorname{{\rm IVF}_{hol}^\omega}}

\newcommand{\Aut}{ \operatorname{{\rm Aut}}}
\newcommand{\Hol}{ \operatorname{{\rm Hol}}}

\newcommand{\C}{\ensuremath{\mathbb{C}}}

\newcommand{\R}{\ensuremath{\mathbb{R}}}
\newcommand{\Q}{\ensuremath{\mathbb{Q}}}
\newcommand{\Z}{\ensuremath{\mathbb{Z}}}
\newcommand{\N}{\ensuremath{\mathbb{N}}}

\newcommand{\B}{\ensuremath{\mathbb{B}}}
\newcommand{\proj}{\ensuremath{\mathbb{P}}}

\newcommand{\bX}{{\bar X}}
\newcommand{\hX}{{\hat X}}
\newcommand{\bY}{{\bar Y}}
\newcommand{\tF}{{\tilde F}}
\newcommand{\tX}{{\tilde X}}

\newcommand{\cB}{{\ensuremath{\mathcal{B}}}}
\newcommand{\cL}{{\ensuremath{\mathcal{L}}}}

\newcommand{\cF}{{\ensuremath{\mathcal{F}}}}
\newcommand{\cG}{{\ensuremath{\mathcal{G}}}}

\newcommand{\cO}{{\ensuremath{\mathcal{O}}}}

\newcommand{\cI}{{\ensuremath{\mathcal{I}}}}
\newcommand{\cM}{{\ensuremath{\mathcal{M}}}}

\newcommand{\Ker}{{\rm Ker} \,}

\newcommand{\Eeul}{\EuScript{E}}

\newcommand{\Oeul}{\EuScript{O}}

\def\ddd{\buildrel {\rm d }\over\longrightarrow}

\renewcommand{\epsilon}{\varepsilon}
\renewcommand{\phi}{\varphi}
\renewcommand{\emptyset}{\varnothing}
\begin{document}
\title[On the present state of the  \textsc{Anders\' en-Lempert} theory ]{On the present state of the  \textsc{Anders\' en-Lempert} theory}

\author{Shulim Kaliman}
\address{Department of Mathematics\\
University of Miami\\
Coral Gables, FL 33124, USA}
\email{kaliman@math.miami.edu}
\author{Frank Kutzschebauch}
\address{Department of Mathematics\\
University of Bern\\
Bern, Switzerland, }
\email{frank.kutzschebauch@math.unibe.ch}

\maketitle

{\em \hspace{2.5in} Dedicated to Professor Peter Russell}

{\em \hspace{2.5in} on the occasion of his seventieth birthday\\[2ex]}

\begin{abstract} In this survey of the
 \textsc{Anders\' en-Lempert} theory we present the state of the art in
 the study of the density property (which means that the Lie algebra generated by completely integrable
 holomorphic vector fields on a given Stein manifold is dense in the space of all holomorphic vector fields). There
 are also two new results in the  paper one of which is
 the theorem stating that the product of Stein manifolds with the volume density
 property possesses such a property as well. The second one is a meaningful
 example of an algebraic surface without the algebraic density property.
 The proof of the last fact requires   \textsc{Brunella}'s technique.

\end{abstract}

{\renewcommand{\thefootnote}{}
\footnotetext{
2000 \textit{Mathematics Subject Classification.} Primary: 32M05,14R20.
Secondary: 14R10, 32M25.}}

\vfuzz=2pt

\vfuzz=2pt
%%%%%%%%%%%%%%%%%%%%%%%%%%%%%%%%%%%%%%%%%%%%%%%%%%%%%%%%%%%%%%%%%%%
%%%%%%%%%%%%%%%%%%%%%%%%%%%%%%%%%%%%%%%%%%%%%%%%%%%%%%%%%%%%%%%%%%%
%%%%%%%%%%%%%%%%%%%%%%%%%%%%%%%%%%%%%%%%%%%%%%%%%%%%%%%%%%%%%%%%%%%

\section{Introduction}

In this paper we discuss recent developments in  the  \textsc{Anders\' en-Lempert}
theory.  %  \cite{FR},),\cite{F}, namely the case of volume preserving maps. Recall that Anders\' en-Lempert
This theory describes complex manifolds such that among other things  the local phase flows on their holomorphically convex compact
subsets can be approximated by global holomorphic
automorphisms which leads to construction of  holomorphic
automorphisms with prescribed local properties. Needless to say that
this implies remarkable consequences for such manifolds some of which are described below
in Section 2. The original work
of   \textsc{Anders\' en} and  \textsc{Lempert}  (\cite{A}, \cite{AL}) established that complex Euclidean spaces of dimension at least 2 belong to
this class. Their results were extended by  \textsc{Forstneri\v c} and  \textsc{Rosay}  \cite{FR}, \cite{Ros99} who discovered new approximation theorems for such spaces.
%(e.g., see \cite{V1}, \cite{V2}, \cite{KK1}).
Perhaps, they understood already that a complex
manifold has such approximations if it possesses the following density property introduced later by \textsc{Varolin}.

\begin{definition}\label{1.10}
A complex manifold $X$ has the density property if in the
compact-open topology the Lie algebra $\LieH (X)$ generated by
completely integrable holomorphic vector fields on $X$ is dense in
the Lie algebra $\VFH (X)$ of all holomorphic vector fields on
$X$. An affine algebraic manifold $X$ has the algebraic density
property if the Lie algebra $\LieA (X)$ generated by completely
integrable algebraic vector fields on it coincides with  the Lie
algebra $\VFA (X)$ of all algebraic vector fields on it.
\end{definition}

Usually the  easiest way to establish the density property for an affine algebraic manifold $X$ is to prove the algebraic density property for it
since some convenient tools from affine algebraic geometry are available in this case\footnote{
Furthermore, the authors do not know examples of such an $X$ with the density property but without the algebraic density property.}.
In the sequence of papers   (\cite{V1}, \cite{TV1}, \cite{TV2},  \cite{KK1},  \cite{KK2}, \cite{DoDvK})
the algebraic density property was established for a wide variety
of affine algebraic manifolds, including all homogeneous spaces that coincide with quotients of linear algebraic groups
with respect to their reductive subgroups
provided that the connected components of these quotients are different from $\C$ or complex tori. The last two papers were based on a very effective
criterion of the algebraic density property introduced by the authors which will be presented in Section 3.

\textsc{Anders\'en, Lempert, Forstneri\v c, Rosay}, \textsc{Toth, Varolin}, and the authors
 considered also another property which has similar consequences for automorphisms preserving a
volume form.
\begin{definition}\label{1.20}
Let a complex manifold  $X$ be
equipped with a holomorphic volume form $\omega$
(i.e. $\omega$ is nowhere vanishing section of the canonical
bundle). We say that $X$ has the volume density property with
respect to $\omega$ if in the compact-open topology the Lie
algebra $\LieHO$ generated by completly integrable holomorphic
vector fields $\nu$ such that $\nu (\omega)=0$, is dense in the
Lie algebra $\VFHO (X)$ of all holomorphic vector fields that
annihilate $\omega$ (note that condition $\nu (\omega)=0$ is
equivalent to the fact that $\nu$ is of $\omega$-divergence zero).
If $X$ is affine algebraic we say that $X$ has the algebraic
volume density property with respect to an algebraic volume form $\omega$ if the Lie
algebra $\LieAO$ generated by completely integrable algebraic vector
fields $\nu$ such that $\nu (\omega)=0$, coincides with the Lie
algebra $\VFAO (X)$ of all algebraic vector fields that annihilate
$\omega$.

\end{definition}

It is much more difficult to establish  the algebraic volume density property than
the algebraic density property since the criterion mentioned before is not applicable
in the volume case. However \textsc{Anders\'en} \cite{A} established the algebraic volume
density property for Euclidean spaces even before the algebraic density property. Some extra
manifolds with the algebraic volume density were found by \textsc{Varolin} and eventually the
authors \cite{KK2} proved it for all  linear algebraic groups (with respect to the left or right invariant volume forms).
Furthermore, they established some features that are straightforward for the algebraic density property
and not at all clear in the volume-preserving case. For instance, they
prove that the algebraic volume density property for an affine algebraic manifold $X$
implies the volume density property for such an $X$ and that the product of two manifolds with algebraic volume
density property has again the algebraic volume density property. Some facts on this subject are contained
in Section 4 together with the following new result.

\begin{theorem}\label{1.30} Let $X$ and $Y$ be Stein manifolds equipped with (holomorphic) volume
forms $\omega_X$ and $\omega_Y$ respectively. Suppose that $X$ and $Y$ have the volume density
property with respect to these forms. Then so does $X \times Y$ with respect to the form $\omega_X \times \omega_Y$.

\end{theorem}

Unlike in the algebraic case the proof of this theorem required some nontrivial facts from functional analysis.

In Sections 5-8 we consider a new meaningful example of an affine surface without the algebraic density property.
The algebraic density property for an affine  algebraic algebraic manifold implies that this manifold has an $m$-transitive group of {\em holomorphic}
automorphisms for any natural $m$ (i.e. every $m$-tuple of distinct points in the manifold can
be transformed into any other such tuple by an automorphism). Thus it makes sense to consider such objects only.
The question about examples  of $m$-transitive affine algebraic manifolds (for any $m$) without the algebraic density
property were posed  by   \textsc{P. Russell} and  \textsc{D. Akhiezer} to the first author who was not aware at the time
about the third paper of \textsc{Anders\'en} \cite{A00}   showing that tori belong to this class. However his proof is
based heavily on a Borel theorem from the Nevanlinna theory which is quite specific for tori. The question whether
there are examples besides tori and the line was still open.

We shall show that the surface $S$ given by $x+y+xyz=1$ in $\C^3$ does not have the algebraic density property. The proof will be based
on remarkable  \textsc{Brunella}'s technique  \cite{Bru98}, \cite{Bru04}, \cite{Brun}, that allowed him to classify completely integrable
algebraic vector fields on the plane (which is based in turn on the seminal preprint of  \textsc{McQuillan}
\cite{McQ} and earlier work of  \textsc{Suzuki} \cite{Su77a}, \cite{Su77b}).
In fact, his technique works for other rational affine surfaces (including the tori $(\C^*)^2$) and, therefore, is more productive than
the \textsc{Anders\'en}'s approach  in the two dimensional case.
We present crucial ingredients of his method in Section 5-7 and apply them to $S$ in Section 8.

{\em Acknowledgements.} It is a pleasure to express our gratitude to  \textsc{Nahum Zobin}
for explaining us important facts from functional analysis.

\section{Applications of the \textsc{Anders\' en-Lempert} theory}

Constructions of holomorphic automorphisms of Stein manifolds with prescribed behavior on compact subsets
are based on the next central theorem of the \textsc{Anders\'en-Lempert} theory which was proven in
the papers \cite{A} and  \cite{AL}  of  \textsc{Anders\'en}   and \textsc{Lempert}  for Euclidean spaces.
We give a stronger version of this theorem which   is  due to \textsc{Forstneri\v c} and \textsc{Rosay} \cite{FR}.
They  considered
it also in the case Euclidean spaces only; however it was essentially their contribution that made the
original result an important tool as it is now.

\begin{theorem} \label{AL-Theorem} Let $X$ be a Stein manifold with the density (resp.  volume density) property
and let  $\Omega$ be an open subset of $X$. Suppose that  $ \Phi : [0,1] \times \Omega \to X$ is a $C^1$-smooth
map  such that

{\rm (1)} $\Phi_t : \Omega \to X$ is holomorphic and injective  (and resp. volume preserving) for every  $ t\in [0,1]$,

{\rm (2)} $\Phi_0 : \Omega \to X$ is the natural embedding of $\Omega$ into $X$, and

{\rm (3)} $\Phi_t  (\Omega) $ is  a Runge
%a holomorphically convex
subset\footnote{Recall that an open subset $U$ of $X$ is Runge if any holomorphic function on $U$ can
be approximated by global holomorphic functions on $X$ in the compact-open topology. Actually, by Cartan's Theorem A this definition
implies more: for any coherent sheaf on $X$ its section over $U$ can be approximated by global sections.}  of $X$ for every $t\in [0,1]$.

Then for each $\epsilon >0 $ and every compact subset $K \subset \Omega$ there is a continuous family,
$\alpha: [0, 1] \to \Aut_{hol} (X)$  of  holomorphic (and resp. volume preserving) automorphisms of $X$
such that  $$\alpha_0 = \Phi_0 \, \, \, {\rm and} \, \, \,\vert \alpha_t - \Phi_t \vert_K <\epsilon $$ % \quad  \forall \ t\in [0,1].$$
for every $t \in [0,1]$.
\end{theorem}

Furthermore, approximations on Stein manifolds with the density property can be chosen with some specific features
as in the next result \cite{V2}.

\begin{proposition}\label{appli.10} Let $X$ be a Stein manifold of dimension $n\geq 2$ with the density
(resp. volume density) property, $K$ be a compact
in $X$, and $x, y \in X$ be two points outside the convex hull of $K$. Suppose that $x_1, \ldots , x_m \in K$.

Then there exists a (resp. volume-preserving) holomorphic automorphism $\Psi$ of $K$ such that $\Psi (x_i) =x_i$ for every $i=1, \ldots , m$,
$\Psi |_K : K \to X$ is as close to the natural embedding as we wish, and $\Psi (y) =x$.
%\footnote{Moreover,
%this automorphism approximating the natural embedding on $K$ can be chosen with prescribed jets at the points
%$x_1, \ldots , x_m, y$.}

\end{proposition}

\begin{remark}\label{appli.15}  This Proposition \ref{appli.10} was proven in \cite{V2} for the density property only
but a slightly modified argument works  in the volume-preserving case. More precisely, let $\gamma$ be a piece-wise analytic path
between $y$ and $x$ that does not meet the holomorphic hull $K_0$ of $K$,  $K_1$ be the union
of $K_0$ and a small ball $U$ around $y$, and $\gamma_1= U \cap \gamma$.  The main step in construction of $\Psi$, where
the adjustment is needed, is a global
approximation on $X$ of a holomorphic vector field $\nu$ on $K_1$ which is identically zero on $K_0$ and which
is tangent to $\gamma_1$. In the volume-preserving case
not only $\nu$ but also its global approximation must be of divergence zero which prevents the direct use
of the Runge property as in \cite{V2}. However, it turns out that the existence of such an approximation is
equivalent  to approximation of some closed holomorphic $(n-1)$-form $\alpha$ on $K_1$ that is identically zero on $K_0$
by a global closed holomorphic $(n-1)$-form on $X$ (see the proof of Claim in Lemma \ref{4.50} for more accurate details).
Note that $\alpha$ is exact (indeed it is zero on $K_0$ and $H^{n-1} (U,\C )=0$ which enables us to apply de Rham's theorem),
i.e.  $\alpha = {\rm d} \beta$ where $\beta$ is a holomorphic $(n-2)$-form on $K_1$.
Since $K_1$ is still holomorphically convex, the Runge property implies that one can approximate $\beta$ by a global holomorphic $(n-2)$-form $\beta'$
which yields
an approximation  of $ \alpha$ by the global closed $(n-1)$-form ${\rm d} \beta'$.
We want to emphasize that this reasoning uses the assumption that  $n\geq 2$ which is essential for the volume-preserving case.
For $n=1$, consider, for example, $X=\C^*$ which has the
volume density property with respect to the volume form ${\rm d} z/z$ where $z$ is the coordinate on $\C^*$. However it does not
satisfy this analogue of Proposition  \ref{appli.10}. By the same reason it has no properties (B)-(E) described below.
\end{remark}

Treating $X$ as a Stein $n$-dimensional manifold in the rest of this section, let
us describe some general properties of Stein manifolds with density (resp. volume density) property which
follow directly from the density property and from subtle applications of Theorem \ref{AL-Theorem} and Proposition \ref{appli.10}. \\

{\bf (A)} If $X$ has the  density (resp. volume density) property, then there are finitely
many completely integrable holomorphic vector fields (resp. of divergence zero)
that span the tangent space at each point (see \cite{KK1}
and Lemma \ref{4.50} below). Therefore
$X$ admits a spray\footnote{There is a small inaccuracy in \cite{KK1}
where the authors proved the existence of a spray. Namely, the metric on the space of
holomorphic automorphisms of $X$ should be defined not as suggested
in that paper but by formula (4.1) below.}, i.e. it is elliptic in the sense of \textsc{Gromov} which implies, in particular, the Oka-Grauert-Gromov principle
for  submersions over Stein spaces with fibers isomorphic to $X$.

To be more precise let us give the relevant definitions and results.
\begin{definition}
%\noindent
{\bf (a)} A (dominating) spray on a complex manifold $X$ is a
holomorphic vector bundle $\rho : E \to X$, together with a
holomorphic map $s : E \to X$,
such that $s$ is identical on the zero section $X \hookrightarrow
E$, and for each $x \in X$ the induced differential map sends the
fibre $E_x=\rho^{-1}(x)$
(which is viewed as a linear subspace of $T_x E$) surjectively
onto $T_x X$.

\smallskip\noindent
{\bf (b)} A fiber-dominating spray for a surjective submersion $h : Z
\to W$ of complex spaces is a vector bundle $\rho : E \to Z$
together with a map $s : E \to Z$ identical on the zero section $Z
\hookrightarrow E$ and such that $h \circ s=h \circ p$ and for
every $z \in Z$ the induced differential map sends
$E_z=\rho^{-1}(z)$ (which is viewed as a linear subspace of $T_z
E$) surjectively onto the subspace of $T_z Z$ tangent to the fiber
$h^{-1}(h(z))$.

\smallskip\noindent
{\bf (c)} Let $ h : Z \to W$ be a holomorphic submersion of Stein spaces,
and ${\rm Cont} (W,Z)$ (resp. ${\rm Holo} (W,Z)$) be the set of
continuous (resp. holomorphic) sections of $h$.
Then $h$ satisfies the Oka-Grauert-Gromov principle if ${\rm Holo}
(W,Z) \hookrightarrow {\rm Cont} (W,Z)$ is a weak homotopy
equivalence. That is,

(i) each continuous section $f^0 : W \to Z$ of $h$ can be
deformed to a holomorphic section $f^1 : W \to Z$, and

(ii) any two homotopic holomorphic sections are also homotopic
through holomorphic sections.

\end{definition}

{\bf Theorem.} (Oka-Grauert-Gromov-principle for elliptic submersions \cite{GromovOPHSEB}, \cite{ForstnericOPSSFB})
 {\em Suppose that $h : Z \to W$ is a
holomorphic submersion of a complex space $Z$ onto a Stein
space $W$ for which every $x \in W$ has a neighborhood
$U\subset W$ such that $h^{-1}(U)\to U$ admits a fiber-dominating
spray. Then $h: Z \to W$ satisfies the Oka-Grauert-Gromov
principle.}

To illustrate this principle recall the following result of the authors \cite{KK1}.

\begin{theorem} \label{modification}   The density (resp. algebraic density)
property holds for smooth analytic (resp. algebraic) hypersurfaces in
$\C^{n+2}_{u,v,{\bar x}}$ given by equations of form $uv=p({\bar
x})$.
\end{theorem}

Hence the Oka-Grauert-Gromov principle for submersions  implies the following.

\begin{corollary} Let $h : Z \to W$ be a surjective submersion of complex
manifolds such that $W$ is Stein and for every $w_0 \in W$ there
is a neighborhood $U$ for which $h^{-1}(U)$ is naturally
isomorphic to a hypersurface in $\C^{n+2}_{{\bar x},u,v}\times U $
given by $uv=p({\bar x}, w)$ where $p$ is a holomorphic function
on $\C_{\bar x}^n \times U$ (independent of $u$ and $v$).

Suppose, furthermore, that $p^*(0) \cap (\C_{\bar x}^n \times w)$
is a smooth reduced proper (may be empty) submanifold of
$\C^n_{\bar x} \times w$ for every point $w \in U$.

Then $h$ satisfies the Oka-Grauert-Gromov principle.
\end{corollary}

{\bf (B)} If $X$ has the density property (or the volume density property) and the
dimension $n$ of $X$ is at least 2, then
the holomorphic automorphisms group $\Aut_{\rm hol} (X)$ acts $m$-transitively on $X$
for any natural number $m$. This was mentioned by \textsc{Varolin} \cite{V2} for the density property
as a simple consequence of Proposition \ref{appli.10}. For the volume density property
this works also due to Remark \ref{appli.15}.\\

{\bf (C)} If $X$ has the density property (or the volume density property) and  the dimension $n\geq 2$
then for each point $x \in X$ there is an injective but not surjective holomorphic map $f : X \to X$ with $f(x) = x$.
The images of such maps are called Fatou-Bieberbach-domains of the second kind.
This was also observed by \textsc{Varolin} in \cite{V2} for the density property, for the volume density property
this is an equally simple application of the kick-out method of \textsc{Dixon} and \textsc{Esterle} \cite{DE}.\\

{\bf (D)} If $X$ has the density property, then for each point $x\in X$  there is an injective non-surjective equidimensional holomorphic map $f : \C^{n} \to X$ with $f(0) = x$.
This observation is due to \textsc{Varolin} \cite{V2}. In particular, all Eisenman measures on $X$ vanish identically.
Such maps are called Fatou-Bieberbach maps of the first kind and
their images are Fatou-Bieberbach domains.\\

\begin{remark} \label{basin} Here is a sketch of the proof of the last fact.
Take a holomorphically convex neighborhood $\Omega$ of $x$ together with a vector field $\theta$ on $\Omega$
for which $x$ is an attractive point.
Approximating the flow of this field by
automorphisms of $X$ as in Theorem \ref{AL-Theorem} we obtain an automorphism whose restriction to a neighborhood of $x$ is
a contraction (in some metric) to a point near  $x$. The basin of attraction of $x$ for that automorphism will be
biholomorphic to $\C^n$. Clearly $\theta$ cannot be volume preserving when contracting to a point.  Therefore the volume density
property does not guarantee the existence of Fatou-Bieberbach domains. Say, it is known that
$(\C^*)^n$ has the volume density property but for $n\geq 2$ it is  an open question whether it contains a Fatou-Bieberbach domain or not.
In particular, it is unknown whether it has the density property as well.
\end{remark}

{\bf (E)} If $X$ has the density (resp. volume density) property
and $Y$ is  any manifold which admits a proper holomorphic embedding $\varphi : Y \hookrightarrow X$ then the following is true.
For any given discrete subset $E=\{ x_1, x_2, \ldots , x_m, \ldots \}$ in $X$ there is another proper holomorphic embedding $\psi$ of $Y$ into $X$
whose image contains $E$. The original proof of this fact was obtained by
\textsc{Globevnik, Forstneri\v c} and \textsc{Rosay} \cite{FGR} in the case of $X=\C^n, \, n\geq 2$ but it works also in
the general case. This proof is based on Proposition \ref{appli.10}. Namely, if the image
of a proper holomorphic embedding $\varphi_m : Y  \hookrightarrow X$ contains already $x_1, \ldots , x_m$ but not $x_{m+1}$, one has to chose a compact
$K_m$ containing $x_1, \ldots , x_m$ so that $x_{m+1}$ is outside the holomorphic hull of $K_m$. Setting $K=K_m$ and
taking $y$ outside the convex hull we construct an automorphism $\Psi=:\Psi_m$ as in Proposition \ref{appli.10}
and replace $\varphi_m$ by $\varphi_{m+1} =\Psi_m \circ \varphi_m$. Then the image of $\varphi_{m+1}$ contains already $x_1, \ldots , x_{m+1}$.
It turns out that compacts $K_m$ and automorphisms $\Psi_m$ can be chosen so that the limit $\psi = \lim_{m\to \infty} \varphi_m$ is also
a proper holomorphic embedding which implies the desired conclusion.

It is also worth mentioning that by a result of \textsc{Winkelmann} \cite{W} (generalizing the
earlier results of \textsc{Rosay} and \textsc{Rudin} \cite{RR1} and the first author \cite{K}), there is a discrete
subset $E$ of $X$ whose complement is $n$-Eisenman hyperbolic. Hence there exists a proper embedding
$\psi : Y \hookrightarrow X$ with
$n$-Eisenman hyperbolic complement $X \setminus \psi (Y)$.\\

We continue now the list of further applications of Theorem \ref{AL-Theorem} in the case where $X$ is
Euclidean space $\C^n, n>1$.\\

{\bf (1)} Property (E) yields a counterexample to the analytic version of the
Abhyankar-Moh-Suzuki theorem which states that every polynomial embedding of $\C$
into $\C^2$ is rectifiable, i.e. the image can be sent to a coordinate line by a polynomial automorphism
of $\C^2$. However, there exists a proper non-rectifiable holomorphic embedding $\psi: \C
\hookrightarrow \C^2$ \cite{FGR}. Indeed, according to (E) we can make $\C^2 \setminus \psi (\C )$ $2$-Eisenman hyperbolic
while the complement $\C^* \times \C^2$ to the coordinate line is not.   In fact, the complement  $\C^2 \setminus \psi (\C )$
can be made Kobayashi hyperbolic. \\

Furthermore, there exist uncountably many non-rectifiable embeddings in a reasonable sense. Let us be more precise.

\begin{definition}\label{def-eq-emb}
Two proper holomorphic embeddings $\Phi,\Psi\colon X\hookrightarrow\C^n$ are {\it equivalent} if there
exist holomorphic automorphisms $\varphi\in\Aut_{\rm hol}(\C^n)$ and $\psi\in\Aut_{\rm hol} (X)$ such that
$\varphi\circ\Phi=\Psi\circ\psi$.
\end{definition}

\begin{remark}\label{appli.25} We would like to emphasize that there is another (weaker) definition of
equivalence. It is so-called $\Aut(\C^n)$-equivalence which was used by several
authors (e.g., \textsc{Buzzard, Forstneri\v c, Globevnik} and \textsc{Varolin}) who proved
uncountability of certain equivalence classes of embeddings in this weaker
sense. In our
definition the map $\Psi^{-1}\circ\varphi\circ\Phi$ is well-defined
and it is an automorphism of $X$ while for the weaker notion one has to
demand that $\Psi^{-1}\circ\varphi\circ\Phi$ is the identity on $X$.
\end{remark}

The best known results are the following.\\

{\bf (2)} In \cite{BK} \textsc{Borell} and the second author showed that
if
(i) $X$ is a Stein space such that its group of holomorphic
automorphisms is a Lie group (with possibly countably many components) and
(ii) there exists a proper holomorphic embedding of $X$ into $\C^m$ where $0<\dim X=n <m$,
then for any $k\geq 0$ there are uncountably many non-equivalent proper
holomorphic embeddings $\Psi\colon X\times\C^k\hookrightarrow\C^m\times\C^k$.\\

{\bf (3)}  These non-equivalent embeddings appear even in holomorphic families.
Under the same assumptions on $X$ as in {\bf (2)} with $k=m-n-1$ there exists
a family of holomorphic embeddings of $X\times
\C^l$ into $\C^m\times \C^l$ parameterized by $\C^k$, such that
for different parameters $w_1\neq w_2\in \C^k$ the embeddings
$\psi_{w_1},\psi_{w_2}:X\times \C^l \hookrightarrow \C^{n+l}$ are
non-equivalent. This result is due to \textsc{Lodin} and the second author \cite{KL}
(an important ingredient of their proof
is a parametric version of Theorem \ref{AL-Theorem}  from \cite{K1}).\\

It is worth mentioning that the last two results include embeddings of $\C^n$ into $\C^m$ for
any $n<m$ (respectively $n<m-1$ for the families), by choosing $X=\C^n$. Thus the holomorphic
analogue of the \textsc{Abhyankar} and \textsc{Sathaye} problem has a negative answer.\\

An important application of the non-rectifiable embeddings is the construction of non-linearizable
holomorphic actions of reductive Lie groups on affine spaces by \textsc{Derksen} and the second author
\cite{DK1}, \cite{DK2}.

\begin{definition}
A holomorphic action of a reductive group $G$ on $\C^n$ is said to be linearizable
if there exists a holomorphic automorphism $\alpha \in \Aut_{\rm hol} (\C^n)$,
such that $\alpha \circ g\circ \alpha^{-1} \in
GL_n ({\C})$ for every $g \in G$.
\end{definition}

{\bf (4)} For any nontrivial complex reductive Lie group $G$ there is a natural $N$ such that for all
$n\ge N$ there is a non-linearizable holomorphic $G$-action on $\C^n$.
The optimal dimension (minimal $N$) is not known for any $G$ including
$G =  \C^*$. All holomorphic actions on $\C^2$ are linearizable by a result of \textsc{Suzuki}
\cite{Su77a} but starting with dimension $n = 4$ there are
non-linearizable actions on $\C^n$. The problem of linearization of holomorphic
$\C^*$-actions on $\C^3$ is still open while all algebraic $\C^*$-actions on $\C^3$
are known to be linearizable \cite{KKMR}.\\

\subsection{Sketch of a construction of a non-linearizable holomorphic $\C^*$-action on $\C^4$}
Suppose $\varphi : \C \to \C^2$ is a proper holomorphic embedding. Consider a pseudo-affine
modification $X$ of $\C_{x, y, u}^3$ along the divisor $D := \{ u = 0\} = \C^2\times \{ 0\}$ with center
$\varphi (\C) \subset D$. That is, if $f \in \Hol (D)$ is a holomorphic
function generating the principal ideal of functions vanishing
on $\varphi (\C)\subset D$, then $X$ is biholomorphic to the submanifold of $\C^4_{x,y,u,v}$ given by the equation $ f(x,y) = u v$.
One of crucial facts observed by Asanuma is that $X  \times \C$ is biholomorphic to $\C^4$ \cite{Asa},
because $X  \times \C$ is the pseudo-affine modification of $\C^4$ along the divisor
$\C^3 \times \{ 0\}$ with center $\varphi (\C)  \times \{ 0\} \subset \C^3$
and any proper holomorphic
embedding of $\C$ into $\C^n$ with the image contained in a hyperplane is rectifiable (see also \cite{KK4}).
Consider the $\C_\lambda^*$-action on $X \times \C_w$ given by
$ \lambda (x, y, u, v, w) = (x, y, \lambda^2 u, \lambda^{-2} v, \lambda w )$.
The categorical quotient $ \C^4 // \C^*_\lambda$ of this action  is equipped with a natural so-called Luna stratification for which one of the strata isomorphic to $\C$
is contained in a higher-dimensional
stratum isomorphic to $\C^2$ exactly in the same manner as  $\C \simeq \varphi (\C)$ is contained in $D\simeq \C^2$.
In the case of a linearizable action this first stratum must be rectifiable in the second one. However by {\bf (1)} we can suppose that
$\varphi (\C )$ is not rectifiable in $D$ which yields a non-linearizable holomorphic $\C^*$-action.

\begin{remark}
The linearization problem for $\C^*$-actions on $\C^3$ is related to the question whether
the pseudo-affine modification $X$ of $\C^3$ as before is biholomorphic to $\C^3$.
If the answer is positive we had a non-linearizable action on $\C^3$,
otherwise it is a counterexample to the holomorphic analogue of Zariski's cancellation problem
(i.e. the question whether a complex manifold $Y$  is biholomorphic
 to $\C^n$ provided that $Y \times \C^k$ is biholomorphic to $\C^{k+n}$).
 Returning to
 $X$ we note that it is diffeomorphic to $\R^6$ as a smooth real manifold
 (\cite[Appendix]{KK1})
and it has the density property by Theorem \ref{modification}. By a conjecture of
\textsc{Varolin} and \textsc{Toth} \cite{TV1} such an $X$ must be biholomorphic to $\C^3$, i.e. in the
case of a negative answer we can disprove their hypothesis.
More potential counterexamples to the conjecture of \textsc{Varolin} and \textsc{Toth} can be found in
\cite{KK1}. One of the most interesting among them is a modification whose center is  the \textsc{Russell} cubic (more precisely,
this modification is isomorphic to the
algebraic hypersurface in $\C^6$ given by the equation $x + x^2 y + s^2 + t^3 = u v$). It is again diffeomorphic to
$\R^{10}$ and has the density property but it is even unknown whether it is different from $\C^5$ as an algebraic variety.

\end{remark}

{\bf (5)} Similar reasoning as before leads from families of holomorphic embeddings like in {\bf (3)}
to families of pairwise non-equivalent $\C^*$-actions. For example,
there is a family $\C_w \times \C^* \to \Aut_{\rm hol} (\C^5)$ of holomorphic $\C^*$-actions on $\C^5$
parametrized by $w\in \C$ such that for different parameters the actions
are non-equivalent (i.e. they are not conjugated by an automorphism). Moreover there is a family such that the
$w = 0$ represents a linear action.
It follows also from {\bf (2)} that on $\C^4$ there are uncountably many non-equivalent $\C^*$-actions.\\

{\bf (6)} One of the questions coming from complex dynamical systems is description of the boundaries of Fatou-Bieberbach domains.
Say, a surprising result of \textsc{Stens\"ones} \cite{S}) provides such a domain in $\C^2$ with a smooth boundary which has, therefore, Hausdorff dimension $d=3$.
Furthermore, it was established by methods of complex dynamical systems that such a dimension can take any value $3 \leq d <4$.
However the question about a Fatou-Bieberbach domain in $\C^2$ with a boundary of Hausdorff dimension $d=4$ remained open until
 \textsc{Peters} and \textsc{Forn\ae ss-Wold} \cite{PW} managed to construct it using  the \textsc{Anders\' en-Lempert}  theory. \\

{\bf (7)} All Fatou-Bieberbach domains arising as basins of attraction (as indicated in Remark \ref{basin})
or more generally as domains of convergence of sequences of automorphisms of $\C^n$ are always Runge domains.
Thus it is natural to ask whether all Fatou-Bieberbach domains in $\C^n$ have to be Runge.
This problem was solved by \textsc{Forn\ae ss-Wold} who constructed a Fatou-Bieberbach domain in $\C \times \C^*$ which is not Runge in $\C^2$
(but Runge in $\C \times \C^*$) using the density property of $\C \times \C^*$ \cite{W1}.\\

{\bf (8)} Developing the ideas from {\bf (7)} further \textsc{Forn\ae ss-Wold} constructed also a ``long $\C^2$" which is not
biholomorphic to $\C^2$, thus solving a classical open question. By a ``long $\C^2$" we mean a complex manifold $X$
which can be exhausted by open subsets $\Omega_i$ which are all biholomorphic to $\C^2$,
i.e.  $X = \bigcup_{i=1}^\infty \Omega_i$, $\Omega_i \subset \Omega_{i+1} $, and $\Omega_i \cong \C^2$  for all $ i \in \N$. Here of course
 $\Omega_i \subset \Omega_{i+1}$ is not a Runge pair.\\

{\bf (9)} A beautiful combination of differential-topological methods with hard analysis (solutions of $\bar\partial$-equations with exact estimates) and the
Anders\'en-Lempert-theory is required for understanding of
how many totally real differentiable  embeddings of a real manifold $M$ into $\C^n$ can exist.

 If $f_0,f_1\colon M\to{\bf C}^n$ are two totally real, polynomially convex real-analytic embeddings of a compact manifold
$M$ into ${\C}^n$, we say that $f_0$ and $f_1$ are $\Aut_{\rm hol} (\C^n)$-equivalent\footnote{It is unfortunate that
in the literature the term   ``$\Aut_{\rm hol} (\C^n)$-equivalence" is used in different meanings - another one
was mentioned in Remark \ref{appli.25}.   } if $f_1=F\circ f_0$, where
$F\colon U\to F(U)\subset{\bf C}^n$ is a biholomorphism defined in a neighbourhood $U$ of $f_0(M)$ such that $F$ is the uniform limit in $U$
of a sequence of elements of $\Aut_{\rm hol}  (\C^n)$. Conditions for $\Aut_{\rm hol}  (\C^n)$-equivalence were found in \cite{FR},
using volume-preserving automorphisms  (and an approach using automorphisms preserving the holomorphic symplectic form
was considered in \cite{F1}).

In the smooth case let  ${\Eeul}^r(M,{\bf C}^n)$ be the set of all totally real polynomially convex $C^r$-embeddings of $M$ into ${\C}^n$
(for $2\le r\le\infty$). It is proved  by \textsc{Forstneri\v c} and \textsc{L\"ow} that two embeddings $f_0, f_1\in{\Eeul}^\infty(M, {\C}^n)$ belong
to the same connected component (in the space of $C^r$-embeddings of $M$ into $\C^n$ equipped with the usual topology
of uniform convergence of all derivative up to order $r$) if and only if there exists a sequence $\{\Phi_j\}\subset{\Aut_{\rm hol} }({\C}^n)$ such that
$\Phi_j\circ f_0\to f_1$ and $\Phi^{-1}_j\circ f_1\to f_0$ in $C^\infty(M)$ as $j\to \infty$. Precise results in the case $r<\infty$ were obtained in  \cite{FLO}.\\

{\bf (10)}  Another interesting problem is to find embeddings of Stein manifolds into $\C^m$ with
prescribed interpolation condition on discrete subsets. This is much more difficult than just letting the
image contain a given discrete subset (not caring about the preimage points). Beside Theorem \ref{AL-Theorem} some properties of
$\C^m$ were used by \textsc{Forstneri\v c, Ivarsson, Prezelj} and the second author in \cite{FIKP} to prove the result below
(for general targets with the density property the problem is still widely open).\\

Let $X$ be a Stein manifold of dimension $n>1$,
$\{a_j\}_{j\in\N}$ (resp. $\{b_j\}_{j\in\N}$) be a sequence of distinct points in $X$ (resp. $\C^m$).
If $m\ge N=\left[\frac{3n}{2}\right] +1$ \footnote{Such $N$ is chosen because it is the optimal embedding dimension,
see {\bf (14)} below.} then there exists a
proper holomorphic embedding $f\colon X\hookrightarrow \C^m$ satisfying
	$$f(a_j) = b_j \, \, \, {\rm for} \, \, \, j=1,2,\ldots .$$
A more general result is that if $X$ admits a proper holomorphic embedding into $\C^m$ and $\{a_j\}_{j\in\N}$ and $\{b_j\}_{j\in\N}$ are
sequences as before such that $\{b_j\}_{j\in\N}$  form a so-called tame subset in $\C^m$
(by definition there is a holomorphic automorphism of $\C^m$ mapping this sequence onto the set of the integer points in a coordinate axis), then the conclusion remains true.
Other results in this direction can be found in \cite{K2}.\\

{\bf (11)} A question  posed by  \textsc{Siu} asks whether there exists always a Fatou-Bieberbach domain contained in the complement
to a closed algebraic subvariety $Z$ of $\C^n$ such that $\dim Z \leq n-2$.
The affirmative answer was obtained by \textsc{Buzzard} and \textsc{Hubbard} who used some concrete construction.
Another proof of this fact was given by the authors who used a version of the density property for such complements (see Theorem \ref{2.130} below).
More precisely for any point $x \in \C^n\setminus Z$ there is a Fatou-Bieberbach (i.e. holomorphic injective) map
 $ f: \C^n \to \C^n\setminus Z$ with $f(0) =x$ (actually, the crucial fact that guarantees such maps is existence of sprays on
 $\C^m\setminus Z$ which can be extracted from the earlier papers of Gromov \cite{GromovOPHSEB} and Winkelmann \cite{Win}).

In particular all Eisenman measures on $\C^n \setminus Z$ are trivial.
It is worth mentioning that closed analytic subsets of $\C^n$ of codimension $k$ may have $k$-Eisenman
hyperbolic complements. More precisely,  it was shown in \cite{BK} that
if a complex manifold $Y$ admits a proper holomorphic embedding into $\C^n$ then it has also another
proper holomorphic embedding with $(n- dim Y)$ -Eisenman hyperbolic complement to the image (the proof is based on the  Anders\'en-Lempert theory
and a generalized idea from \cite{BFo} ).\\

{\bf (12)} The classical approximation theorem of \textsc{Carleman} states that
for each continuous function $\lambda : \R \to \C$ and a positive continuous function
$\epsilon : \R \to (0, \infty)$ there exists an entire function $f$ on $\C$ such that $\vert f(t) - \lambda (t)\vert < \epsilon (t)$ for every $t\in \R$.

Using the Anders\'en-Lempert-theory together with some explicit shears (which are special automorphisms of $\C^n$ which appear in that theory)
\textsc{Buzzard} and \textsc{Forstneri\v c} \cite{BF} were able to prove a similar result for holomorphic automorphisms of $\C^n$.
Namely, for any proper embedding $\lambda : \R \to \C^n$ of class $C^r$ (where $n\geq 2$ and $r\ge 0$)
and a positive continuous function $\epsilon : \R \to (0, \infty)$ there exists a proper holomorphic embedding
$f: \C \to \C^n$ such that $$ \vert f^{(s)} (t) - \lambda^{(s)} (t)\vert < \epsilon (t) \quad \forall \ t\in \R, \ 0\le s\le r.$$

Actually this fact remains valid under the additional requirement that the embedding satisfies the interpolation property as in {\bf (10)}.
\\

{\bf (13)} If $Z$ is a smooth closed algebraic subvariety of $\C^n$ such that $n > 2 dim Z +1$ then it is known \cite{K0} (see also \cite{Sri}) that
any automorphism of $Z$ extends to an automorphism of $\C^n$. In the holomorphic category the situation is completely different
and exploiting Theorem \ref{AL-Theorem}  the second author showed in \cite{DKW} that
there is a proper holomorphic embedding of $\varphi: \C \hookrightarrow \C^n$
with any $n\geq 2$ (thus the codimension is arbitrarily big) such that the only automorphism of $\C^n$ mapping the image onto itself
is the identical one.\\

{\bf (14)}
The classical theorem of \textsc{Remmert} \cite{R} states that any Stein
manifold $Y$ of dimension $n$ admits a proper holomorphic embedding into
a Euclidean space $\C^N$ of sufficiently high dimension $N$. The optimal $N$ was found
in the papers of \textsc{Gromov, Eliashberg}
\cite{EG} and \textsc{Sch\"urmann} \cite{Sch}; in the case of $n\geq 2$ they proved that $Y$ can be embedded into $\C^{[3n/2]+1}$. This
result is sharp by virtue of examples of \textsc{O.~Forster}
\cite{Fo}, who conjectured that the optimal $N$ was $[3n/2]+1$.  The case of $n=1$ is still open, i.e. it is unknown whether any open Riemann surface admits
a  proper holomorphic embedding into $\C^2$. However, there was a recent breakthrough in this
direction - \textsc{Forn\ae ss Wold}
proved the Forster conjecture for all finitely connected domains in $\C$ and
for all elliptic curves with finitely many holes (that are not punctures) using the
Anders\'en-Lempert theory in a very clever way \cite{W2}, \cite{W3}, \cite{W4}.
Furthermore, \textsc{L\"ow, Fornaess-Wold} and the second author \cite{KLW} showed that one can require additionally
an interpolation condition on discrete subsets as in {\bf (10)}.

\section{Criterion for the algebraic density property.}

\textsc{Toth} and \textsc{ Varolin} established the algebraic density property for some manifolds
including semi-simple complex Lie groups \cite{TV1}, \cite{TV2}.
Their proof follows to a great extend the original ideas of
\textsc{Anders\'en} and \textsc{ Lempert} and is quite complicated. A new approach suggested by
the authors \cite{KK2} lead to the following.

\begin{theorem}\label{2.10} Each linear algebraic
group whose connected component is different from $\C_+$ or
$(\C^*)^k, \, k \geq 1$ has the algebraic density property.

\end{theorem}

The proof was based on the following simple fact.

\begin{theorem}\label{2.20}
Let $X$ be an affine algebraic manifold with a transitive
group $\Aut X$ of algebraic automorphisms and let $\C [X]$ be its
algebra of regular functions.
Suppose that there is a submodule $L$ of the $\C [X]$-module $T$
of all algebraic vector fields on $X$ such that $L \subset \LieA (X)$ and the
fiber of $L$ at some point $x_0\in X$ contains a
generating subset\footnote{ Our notion of a generating subset is milder than usual.
A finite subset $F$ of $T_{x_0}X$ is called a
generating subset if the span of the orbit of $F$ under the action
of the isotropy group $(\Aut X)_{x_0}$ coincides with
$T_{x_0}X$ (say, if $X$ is a simple Lie group then every
nonzero vector is a generating set.)
. } of $T_{x_0}X$.
Then $X$ has the algebraic density property.
\end{theorem}

\begin{proof}
The action of $\alpha \in \Aut X$ maps $L$ onto another $\C
[X]$-module $L_\alpha$.
The sum of such modules with $\alpha$ running over a subset
of $\Aut X$ is again a $\C [X]$-submodule $N$ of $T$.
Let $\cM_x\subset \C [X]$ be the maximal ideal that consists of functions vanishing at $x \in X$. By assumption $N$ can
be chosen so that $N / \cM_{x_0} T$ coincides with $T_{x_0}X=T/\cM_{x_0} T$.
Furthermore, since $X$ is homogeneous with respect to $\Aut X$ we can
suppose that the same is true for every point in $X$. That is, for the $\C [X]$-module $Q=T/N$
and every $x \in X$ we have $Q/\cM_x Q=0$. Thus $Q=0$ and $N=T$ (e.g., see \cite[Exercise II.5.8]{Har}).
Since composition with automorphisms preserves complete
integrability, all element of $N$ are in $\LieA (X)$ which implies the desired conclusion.

\end{proof}

Thus the idea of the proof  of Theorem \ref{2.10} is to catch a nontrivial $\C [X]$-module in the $\LieA (X)$.
In order to demonstrate how to do it we prove
the main observation of the Anders\'en-Lempert
theory.

\begin{proposition}\label{2.30}
For $n \geq 2$ the space $\C^n$ has the algebraic
density property. \end{proposition}

\begin{proof}
Let $x_1, \ldots ,x_n$ be a coordinate system on
$X=\C^n$ and $\delta_i =\partial /\partial x_i$, i.e.
$\Ker \delta_i =\C [x_1, \ldots , {\hat x}_i, \ldots , x_n]$
and therefore
$$\C^{[n]} = {\rm Span} \, \Ker \delta_1 \cdot \Ker \delta_2.$$
Note also that for $f_i \in \Ker \delta_i$ the algebraic vector
fields $f_i \delta_i$ and $x_if_i\delta_i$ are completely integrable on  $\C^n$.
This implies that the field
$$[f_1\delta_1,x_1f_2\delta_2]-[x_1f_1\delta_1,f_2\delta_2]=f_1f_2\delta_2$$
belongs to $\LieA (X)$.
Thus $\LieA (X)$ contains all algebraic fields proportional to
$\delta_2$ which in combination with Theorem \ref{2.20} implies the desired conclusion.
\end{proof}

To transfer this argument to other affine algebraic manifolds we have to use locally nilpotent and semi-simple
derivations instead of partial derivatives.
Recall that an algebraic vector field $\sigma$ on $X$ is locally nilpotent
(LND) if its flow is an algebraic $\C_+$-action. Equivalently, $\sigma$ is a LND if for every $a \in \C [X]$ there
exists natural $n$ for which $\sigma^n (a)=0$. The last algebraic definition enables us to introduce
the degree of any regular element $a\in \C [X]$ with respect to $\sigma$ as $\deg_\sigma (a)=
\min \{ n-1 | \sigma^n (a)=0 \}$.

An algebraic vector field on $X$ is semi-simple if its flow
is an algebraic $\C^*$-action.

\begin{definition}\label{2.40}
Let $\sigma$ be a LND on $X$ and $\delta$ be
either a LND or semi-simple.

Then pair $(\sigma , \delta )$ is called semi-compatible if the
the span of $\Ker \sigma \cdot \Ker \delta$ contains a nonzero
ideal of $\C [X]$.

A semi-compatible pair $(\sigma , \delta )$ is called compatible
if one of the following conditions holds:

(1) there exists $a \in \C [X]$ such that $a \in \Ker \delta$ and
$\sigma (a) \in \Ker \sigma \setminus 0$, i.e. $\deg_{\sigma} (a)
=1$.

(2) both $\sigma$ and $\delta$ are LNDs and there exists $a \in \C [X]$
such that $\deg_{\sigma} (a) =1=\deg_{\delta} (a)$.

\end{definition}

Repeating the argument from Proposition \ref{2.30} with $\sigma$ and $\delta$ instead of
$\delta_1$ and $\delta_2$ we get the following.

\begin{proposition}\label{2.50}
The existence of a compatible pair yields the existence of a nontrivial  $\C [X]$-module in $\LieA (X)$.
\end{proposition}

Now we can formulate our criterion as the following.

\begin{theorem}\label{2.60}
Let $X$ be a smooth affine algebraic variety
with a transitive automorphism group $\Aut X$.
Suppose that there are finitely many pairs of compatible vector
fields $\{ \sigma_i , \delta_i \}$ such that
at some point $x_0
\in X$ vectors $\{ \delta_i (x_0) \}$ form a generating subset of
$T_{x_0}X$.
Then $X$ has
the algebraic density property. \end{theorem}

\begin{example}\label{2.70}

(1) As an obvious application of this theorem we see that  the manifold $X=\C^k\times (\C^*)^l$ with $k\geq 1$ and $k+l\geq 2$ has algebraic density property.

(2) A more interesting case is when $X=SL_2$ (or $PSL_2$). Denote by
$$A =  \left(\begin{array}{ccc}
a_1 & a_2\\
b_1 & b_2 \end{array}\right) $$
an element of  $X$.
Consider the following LNDs on $\C [X]$:
$$
\delta_1=a_1\frac{\partial}{\partial
b_1}+a_2\frac{\partial}{\partial b_2}
$$
$$
 \delta_2=b_1\frac{\partial}{\partial
 a_1}+b_2\frac{\partial}{\partial a_2} \, .
$$

Then $\Ker \delta_1 =\C [a_1, a_2]$ and $\Ker \delta_2 =\C
[b_1,b_2]$, i.e.  ${\rm Span} \,  \Ker \delta_1 \cdot \Ker
\delta_2= \C[X]$ and the pair $(\delta_1, \delta_2)$ is
semi-compatible.

Furthermore, $\deg_{\delta_1}(a)= \deg_{\delta_2} (a) =1$ for
$a=a_1b_2$ and therefore the pair $(\delta_1, \delta_2)$ is compatible.
Hence $SL_2$ and
$PSL_2$ have the algebraic density property. It can be shown in the same manner
that $SL_n $ and $PSL_n $ have also algebraic density property.
\end{example}

\begin{definition}\label{2.80} Let us choose an identification of elements of $SL_2$ with $(2\times 2)$-matrices with determinant 1.
Suppose that  $H_1\simeq \C_+$ (resp. $H_2\simeq \C_+$) is
the unipotent upper (resp. lower) triangular subgroup of $SL_2$.
Then any $SL_2$-action on $X$ generates $H_i$-action on $X$ and,
therefore, a locally nilpotent vector field $\delta_i$.
The pair $(\delta_1 ,\delta_2)$ will be called an associated pair
of LNDs of the $SL_2$-action.
\end{definition}

\begin{theorem}\label{2.90}
Let $X$ be a smooth
complex affine algebraic variety whose group of algebraic
automorphisms is transitive.

{\rm (1)} Suppose that $X$ is equipped with a non-degenerate\footnote{That is, the dimension of general orbits
is 3.} fixed point
free $SL_2$-action. Then the associated pair  $(\delta_1 ,\delta_2)$ is compatible.

{\rm (2)} Suppose that $X$ is equipped with $N$ non-degenerate fixed point
free $SL_2$-actions. Let $\{ \delta_1^k,\delta^k_2 \}_{k=1}^N$ be
the corresponding pairs of associated locally nilpotent vector
fields.

If $\{\delta^k_2 (x_0) \}_{k=1}^N \subset T_{x_0}X$ is a
generating set at some point $x_0\in X$ then $X$ has the algebraic
density property. \end{theorem}

The second statement is, of course, a consequence of (1) and Theorem \ref{2.60}.
The idea of the proof of the first statement can be described as follows. By \textsc{Luna}'s
slice theorem (e.g., see \cite{D}) any closed $SL_2$-orbit $O$ possesses an $SL_2$-invariant neighborhood $U'$
for which there exists a surjective \'etale $SL_2$-equivariant morphism $U'' \to U'$ such that $U''$ is naturally isomorphic
to $SL_2\times_{I_x}V$ \footnote{That is, elements of $U''$ are equivalence classes
in $SL_2 \times V$ given by the relation $(s,v) \sim (sg^{-1}, g\cdot v)$ for $g \in I_x$.}   where $I_x\subset SL_2$ is the isotropy
group of some point $x\in O$ and  $V$ is an $I_x$-invariant subvariety of $X$ (called a slice).
Then the associated locally nilpotent derivations $\delta_1$ and $\delta_2$ generate naturally similar derivations $\delta_1'$ and
$\delta_2'$ ( resp. $\delta_1''$ and $\delta_2''$) on $U'$ (resp. $U''$). The straightforward application of Nullstellensatz shows
that $\delta_1$ and $\delta_2$  are compatible provided that  for every closed orbit $O$ the pair $(\delta_1', \delta_2')$ described  before is compatible.
It was shown in \cite{KK2} that compatibility of $\delta_1'$ and $\delta_2'$ follows from compatibility of $\delta_1''$ and $\delta_2''$.
The proof of compatibility of the associated locally nilpotent derivations for
the natural $SL_2$-action on $SL_2\times_{I_x}V$ (especially in the case when $I_x$ contains $\C^*$) is the most
difficult part which does not work for $SL_2$-actions with fixed points.
Actually, if a non-degenerate $SL_2$-action has a fixed point the pair
 $(\delta_1' ,\delta_2')$  is not compatible.

 As a (non-straightforward) application of Theorem \ref{2.90} we have.

\begin{theorem}\label{2.100}
Let $G$ be a
linear algebraic group and $R$ be its proper reductive subgroup.

Suppose that the connected components of $X=G/R$ are different from
$\C_+$ or $(\C^*)^k$.

Then $X$ has the algebraic density property. \end{theorem}

The difficulty in the proof of the last result lies in checking the assumption of Theorem \ref{2.90}. That is,
one needs to find an $SL_2$-subgroup $\Gamma$ of $G$ such that its natural action on $X$ is fixed point
free and non-degenerate.
This is equivalent to the next fact.

\begin{proposition}\label{2.110}
Let  $G$  and $R$ be as before. Then there exists
an $SL_2$-subgroup $\Gamma$ of $G$ such that

{\rm (1)}  $g\Gamma g^{-1}$ is not contained in $R$ for any $g \in G$ and

{\rm (2)} $g_0\Gamma g_0^{-1}$ meets $R$ at a finite set  for some $g_0 \in G$.

Furthermore $\Gamma$ can be chosen as a principal or sub-regular $SL_2$-subgroup of $G$. \end{proposition}

The proof of Proposition \ref{2.110} is surprisingly non-trivial. First using the Jackobson-Morozov theorem
(e.g., see \cite{Bo}) we can switch from studying the conjugate classes of $SL_2$-subgroups in $G$ to the study
of orbits of nilpotent elements in the Lie algebra of $G$ under the adjoint action.  The orbit
of the largest dimension is called principal and the second largest dimension sub-regular (they are both unique).
To show that an $SL_2$-subgroup corresponding to one of these orbits satisfy conditions (1) and (2) one
needs
to use the Dynkin classification of nilpotent orbits (e.g. see \cite{BaCa1} and \cite{BaCa2})
in combination with some difficult facts about
the Freudental square \cite{LMW}.

Besides Theorem \ref{2.90} (1) we have another useful criterion for compatibility \cite[Proposition 3.9]{KK2}.

\begin{proposition}\label{2.112} Suppose that either $\Gamma = SL_2$ and $H_1$, $H_2$ are from
Definition \ref{2.80} or $\Gamma =H_1 \times H_2$ where $H_1 \simeq \C_+$ and $H_2$ is one of the groups
$\C_+$ or $\C^*$. Let $X$ be a smooth affine algebraic $\Gamma$-variety and $Y$ be a normal affine algebraic
variety equipped with a trivial $\Gamma$-action. Let $r : X \to Y$ be a surjective $\Gamma$-equivariant morphism
and $\delta_1, \delta_2$ be the algebraic fields on $X$ generated by the action of $H_1$ and $H_2$.
Suppose that for every point $y \in Y$ there exists an \'etale neighborhood $W \to Y$ such that
the pair of vector fields induced by $\delta_1$ and $\delta_2$ on the fibered product $X\times_Y W$ is compatible.
Then the pair $(\delta_1, \delta_2)$ is compatible.

\end{proposition}

\subsection{ Non-Stein Case: Complement to a Subvariety of  Codimension at least 2.}

Let $Y$ be an algebraic subvariety of an affine
algebraic manifold $X$ such that ${\rm codim}_X Y \geq 2$.
Then according to \textsc{Forstneri\v c}  $X\setminus Y$ has the algebraic density property if the Lie
algebra $\LieA (X,Y)$ generated by completely integrable algebraic
vector fields on $X$ that vanishes on $Y$ coincides with the Lie
algebra $\VFA (X,Y)$ of all algebraic vector fields on $X$ that
vanishes on $Y$.

\begin{remark}\label{2.120} By Hartogs' theorem every algebraic vector field on $X\setminus Y$
can be extended to $X$. In the case of completely integrable fields this extension must
be tangent to $Y$. In particular, $\LieA (X, Y)$ cannot contain algebraic fields with extension
non-tangent to $Y$. Furthermore, in the case of $Y$ without the algebraic density property
(say, when $Y$ is singular) $\LieA (X, Y)$ cannot contain all vector fields tangent to $Y$.
Thus the \textsc{Forstneri\v c} definition is the only reasonable.

\end{remark}

\begin{theorem}\label{2.130}  {\rm (\cite{KK2})}
Let $Y$ be a
subvariety of $\C^n$ of codimension $\geq 2$. Then

{\rm (1)} $\C^n \setminus Y$ has the algebraic density property if
$T_y Y \leq n-1$ for every $y \in Y$;

{\rm (2)} in any case there exists natural $k$ such that $ \LieA
(\C^n,Y)$ contains all
algebraic vector fields that vanish on $Y$ with multiplicity at least $k$.
\end{theorem}

It can be shown that even under condition (2) the approximation theorems we had in the affine
case are valid for $\C^n \setminus Y$.

\section{ Volume density property}

As we mentioned before, the algebraic volume density property (see Definition \ref{1.20})
is much more delicate thing to prove than the algebraic density property. It is hopeless to
look for a $\C [X]$-module inside the Lie algebra $\LieAO (X)$ generated by completely
integrable algebraic vector fields of divergence zero with respect to an algebraic
volume form $\omega$ on $X$. Indeed, let
$\nu$ be a vector field on $X$ of $\omega$-divergence zero and
$f \in \C [X]$. Then $\diver_{\omega} (f \nu ) = \nu (f)$ is nonzero for general $f$. However,
some important facts remain valid in this volume-preserving case \cite{KK3}.

\begin{theorem}\label{4.10}
Let $X$ be a smooth
hypersurface in $\C_{u,v, \bar x}^{n+2}$ given by an equation
$uv=p({\bar x})$. Suppose that $Z \subset \C_{\bar x}^n$ is given
by $p(\bar x)=0$ and $H^{n-2}(Z,\C )=0$. Then $X$ has the algebraic
volume density property.
\end{theorem}

Note that unlike in Theorem \ref{modification} we have an extra assumption on the zero locus $Z$;
namely $H^{n-2} (Z, \C )=0$. It is unknown whether this assumption is essential.
However this is enough to show that $SL_2$ (given by equation $uv=xy+1$ in $\C^4$) and
$PSL_2$ have the algebraic volume density property with respect to the invariant volume form
which is important in the proof of the next result \cite{KK3}.

\begin{theorem}\label{4.20}
Every linear algebraic group has the algebraic volume density property with respect
to the left-invariant (or right-invariant) volume form.
\end{theorem}

The proof uses also the second statement of the following theorem which is not at all
obvious in the volume-preserving case.

\begin{theorem}\label{4.30}
{\rm (1)} For an affine algebraic manifold $X$ equipped with an algebraic
volume form $\omega$ the algebraic volume density property implies
the volume density property (in the holomorphic sense).

{\rm (2)} If affine algebraic manifolds
$X$ and $Y$ have the algebraic volume density property with respect to volume forms
$\omega_X$ and $\omega_Y$, then so
does $X \times Y$ with respect to the volume form $\omega_X \times  \omega_Y$.
\end{theorem}

{\em Sketch of the Proof of (1).}
Let $\mu$ be a holomorphic vector field such that $\mu (\omega
)=0$. Our aim is to find an algebraic vector field $\nu$ with $\nu (\omega )=0$
that approximates $\mu$.

In the first step one can show that
there is a natural duality between the holomorphic (resp. algebraic) vector fields of
$\omega$-divergence zero and the closed
$(n-1)$-forms on $X$ given by $\mu \to \iota_\mu (\omega)$ where $\iota_\mu$ is
the inner product (see \cite[Lemmas 3.5]{KK3}).

Next one needs to use the remarkable
\textsc{Grothendieck}'s theorem \cite{Gro} that states that the computation of de Rham cohomology
on a smooth affine algebraic variety can be achieved via the complex
of algebraic forms. Thus there exists an algebraic $(n-1)$-form $\tau_{n-1}$
such that $\iota_\mu (\omega) -\tau$ is exact and equal to ${\rm
d} \tau_{n-2}$ where $\tau_{n-2}$ is a holomorphic $(n-2)$-form.
Approximate  $\tau_{n-2}$ by an algebraic form  $\tau_{n-2}'$ and
set $\nu$ equal to the dual of the form  $ \tau_{n-1}+{\rm d} \tau_{n-2}'$. \hspace{4in} $\square$

Instead of proving (2) we present a more complicated proof of its analogue in the holomorphic
case mentioned in the introduction.

\begin{theorem}\label{4.40}
If Stein manifolds
$X$ and $Y$ have the volume density property with respect to forms
$\omega_X$ and $\omega_Y$, so
does $X \times Y$ with respect to the form $\omega_X \times  \omega_Y$.

\end{theorem}

In preparation for the proof of Theorem \ref{4.40} we must establish a number of facts. Recall that for a complex
manifold $X$ and a holomorphic volume form $\omega$ on it the space of holomorphic vectors fields of
$\omega$-divergence zero is denoted by $\VFHO (X)$ and the Lie algebra generated by the set  $\IVFHO (X)$ of completely integrable holomorphic
vector fields from $ \VFHO (X)$ is denoted by $\LieHO (X)$.  First we shall prove the holomorphic analogue of Lemma 4.2 from \cite{KK3}.

\begin{lemma} \label{4.50}
Let $\omega$ be a holomorphic form on a Stein manifold $X$ such that $X$ has the volume density property with respect to $\omega$. Then there exist finitely many
vector fields $\delta_1, \delta_2, \ldots , \delta_N \in \IVFHO (X)$ such that
${\rm Span}\{ \delta_i  (x) | \, i=1, \ldots, N \} = T_x X$ for every point $ x\in X$.

\end{lemma}

\begin{proof}
We start with the following.

{\em Claim.} The set $\VFHO (X)$ generate $T_xX$ for every $x \in X$.

Lets us assume first that the dimension of $X$ is bigger than one.
Let $x \in X$ and $U$ be a Runge neighborhood of $x$ such that $H^{n-1} (U, \C) = 0$ where $n = \dim X$ (take,
for instance, a small sublevel set $U$ of a strictly plurisubharmonic  exhaustion
function on $X$ with minimum at $x$).
Shrinking $U$ we can assume that in some holomorphic coordinate system $z_1, \ldots ,z_n$
on $U$
the form $\omega \vert_U$ is the standard volume  ${\rm d} z_1\wedge  \ldots \wedge {\rm d} z_n$.
Thus the holomorphic vector fields $\partial /\partial z_i$ on $U$ are of divergence zero
and they span the tangent space at $x$. We need to approximate them by global holomorphic
fields of divergence zero on $X$ which would yield our claim.  As in Theorem \ref{4.30}
consider the inner product $\iota_\nu (\omega ) =: \alpha$ for $\nu \in \VFHO (U)$.  By \cite[Lemma 3.5 (1)]{KK3} $\alpha$ is a closed
$(n-1)$-form  on $U$ and since $H^{n-1} (U, \C) = 0$ we  can
find an $(n-2)$-form $\beta$ on $U$ with ${\rm d} \beta = \alpha$.
Since $U$ is Runge in $X$ we can also
approximate $\beta$ by a global holomorphic $(n-2)$-form $\tilde \beta$ (uniformly on compacts in $U$). Then the closed holomorphic $(n-1)$-form ${\rm d} \tilde\beta$
approximates $\alpha$ and the unique holomorphic vector field $\theta$ defined by
$\iota_\theta (\omega ) = {\rm d} \tilde \beta$ approximates $\nu$. Since
${\rm d} \tilde \beta$ is closed, the field $\theta$ is of divergence zero which
concludes the proof of the Claim in the case when the dimenson of $X$ is bigger than one.

Now assume the dimension of $X$ is equal to one, i.e., $X$ is %an open Riemann surface
a smooth Stein curve with the volume density property.
In particular, it admits a non-trivial completely integrable holomorphic vector field and, therefore, a non-constant
holomorphic map $\C \to X$.  This implies that $X$ is isomorphic to either $\C_+$ or $\C^*$. In each of these cases the volume
form $\omega$ for which $X$ has the volume density property is the unique (up to a constant factor) invariant volume form on the corresponding group
and there exists
a nowhere vanishing completely integrable vector field of $\omega$-divergence zero which concludes the proof of the Claim.

It follows from the Claim and  the  volume density property
that  vector fields from $ \LieHO (X)$ span the tangent space $T_xX$  at any given
point $x \in X$. Observe that every Lie bracket $[\nu , \mu ]$ of completely integrable holomorphic  vector fields of divergence zero
can be approximated by a linear combination of such fields which follows immediately from
the equality $[ \nu , \mu ] = \lim_{t \to 0}  {\frac{\phi_t^* (\nu ) - \nu}{t}}$ where $\phi_t$ is
the flow generated by $\mu$. Thus the set $\IVFHO (X)$ generates  $T_x X$ at any $x \in X$.

To prove that there are finitely many fields from $\IVFHO (X)$ that span each tangent space
let us start with $n$ fields $\theta_1, \ldots, \theta_n$ which
span the tangent space at some point $x_0$ and thus outside a proper analytic subset $A$. The set $A$
may have countably many irreducible components $A_1, A_2, A_3, \ldots$.

It suffices now
to find a volume preserving holomorphic automorphism $\Phi \in \Aut_{\rm hol}^\omega (X)$ such that
$\Phi (X \setminus  A) \cap A_i \ne \emptyset$ for every  $ i = 1, 2, 3, \ldots$. Indeed,
for such an automorphism $\Phi$ the completely integrable holomorphic vector  fields $\Phi_*(\theta_1), \ldots,
\Phi_* (\theta_n)$ have divergence zero and span the tangent space at a general point in each $A_i$, i.e. together with the fields
$\theta_1, \ldots, \theta_n$ they span the tangent space at each point outside an analytic subset $B$ of a
smaller dimension than $A$. Then the induction by dimension implies the desired conclusion.

In order to construct $\Phi$ consider a monotonically increasing sequence of compacts $K_1 \subset K_2 \subset \ldots $ in $X$ such
that $\bigcup_i K_i =X$ and a closed imbedding $\iota : X \hookrightarrow \C^m$. For every continuous map $\varphi : X \to \C^m$ denote
by $||\varphi ||_i$ the standard norm of the restriction of $\varphi$ to $K_i$.
Let $d$ be the metric on the space $ \Aut_{\rm hol}  (X)$ of holomorphic automorphisms  of $X$
given by the formula $$d(\Phi , \Psi ) =\sum_{i=1}^\infty 2^{-i}( \min (||\Phi -\Psi ||_{i},1)+ \min (||\Phi^{-1} -\Psi^{-1} ||_{i},1) \eqno{(4.1)}$$
where automorphisms $\Phi^{\pm 1}, \Psi^{\pm 1} \in \Aut_{\rm hol} (X)$  are viewed as continuous  maps from $X$ to $\C^m$.
This metric makes $ \Aut_{\rm hol}  (X)$ a complete metric space. Its subset $  \Aut_{\rm hol}^\omega (X) $ of volume-preserving automorphisms
is closed and, therefore, it is a complete metric space as well.

Set $Z_i = \{ \Psi \in  \Aut_{\rm hol}^\omega (X) : \Psi (A_i) \cap (X\setminus A) \ne \emptyset \} $. Note that $Z_i$ is open
in $ \Aut_{\rm hol}^\omega (X)$ and let us show that it  is also everywhere dense. Indeed,
the flow of any   $\theta \in \IVFHO (X)$  preserves $\omega$ because
for any vector fields $\nu$ its divergence with respect to $\omega$ is defined by the formula $\diver_{\omega} (\nu ) \omega =L_{\nu} (\omega )$ where
$L_\nu$ is the Lie derivative.
Since $\IVFHO (X)$ generates the tangent space
at each point of $X$ we can choose $\theta$ non-tangent to $A_i$. Then for every $\Psi \in  \Aut_{\rm hol}^\omega (X)$  its composition with general
elements of the flow induced by $\theta$ is  in $Z_i$.
That is,  a perturbation of $\Psi$ belongs to $Z_i$ which proves that $Z_i$ is everywhere dense in $ \Aut_{\rm hol}^\omega (X)$.
By the Baire category theorem the set $\bigcap_{i=1}^\infty Z_i$ is not empty which yields
the existence of the desired automorphism.

\end{proof}

\begin{lemma} \label{4.60} For a holomorphic field $\nu$ on $X$
denote by $ \omega_{\nu} $ the inner product $\iota_\nu (\omega )$ of $\nu$ with the form $\omega$.  If $\nu$ is of $\omega$-divergence zero
then for every holomorphic function $f$ on $X$ we have
${\rm d} ( \omega_{f\nu} ) = \nu (f) \omega$.
\end{lemma}

\begin{proof}
Using the well-known formula $ L_{\nu}= {\rm
d}\circ \iota_{\nu} + \iota_{\nu} \circ {\rm  d} $ (e.g., see \cite[Proposition 3.10]{KN1})
that relates the outer differentiation ${\rm d}$ with the Lie derivative $L_\nu$ and inner
product $\iota_\nu$ we get
$${\rm d} ( \omega_{f\nu} ) = L_{f\nu} (\omega )- \iota_{f \nu} {\rm d} (\omega)=
L_{f\nu} (\omega )=\diver_{\omega} (f\nu )\omega = (f\diver_{\omega} (\nu )+\nu (f) ) \omega =\nu (f) \omega .$$
\end{proof}

\begin{notation}\label{4.70} For a complex manifold $X$ with a volume form $\omega$ we define the following linear subspace $F_X$ in the space
of holomorphic functions  ${\rm Hol} (X) = H^0(X,\Oeul (X))$ on $X$:

$$F_X = {\rm Span} \{ \theta ({\rm Hol} (X)) \ : \ \theta \in \VFHO
 (X)\} .$$

\end{notation}

\begin{lemma} \label{4.80}
Let $\Omega^i$ be the sheaf of holomorphic $i$-forms on a complex manifold $X$ and let $\cB^i \subset \Omega^i$
be the subsheaf of exact $i$-forms.
For $n=\dim X$ consider the isomorphism
$\epsilon : {\rm Hol} (X) \to H^0(X, \Omega^n )$ into the space $H^0(X, \Omega^n)$
of holomorphic $n$-forms  given by $\epsilon (f) = f \omega$ and suppose that

{\rm (i)} the ${\rm Hol} (X)$-module $\VFH (X)$ of holomorphic vector fields on $X$ is generated  by $\VFHO (X)$.

Then
$F_X =  \epsilon^{-1} (H^0(X, \cB^n ))$ where $H^0(X, \cB^n)$ is the space of exact $n$-forms.

Moreover,  suppose that  $X$ is Stein and instead of condition (i) we have

{\rm (ii)}  finitely many vector fields $\delta_1, \delta_2, \ldots , \delta_N \in \VFHO (X)$  generating $T_x X$ at each $x\in X$.

Then
$F_X = {\rm Span} \{ \delta_i  ({\rm Hol} (X)) |  \, i=1, \ldots, N \}. $
\end{lemma}

\begin{proof}
Since $\omega$ is nowhere
vanishing we have an  isomorphism between the spaces of holomorphic vector fields and holomorphic $(n-1)$-forms  given by  $\nu \to \iota_\nu (\omega )$,
i.e. any $(n-1)$-form can be presented as $\omega_\nu$ for some $\nu \in \VFH (X)$. Furthermore, condition (i) implies that
the vector space $H^0 (X, \Omega^{n-1})$ is generated by elements of type $\omega_{f\delta}$ where $\delta \in \VFHO (X)$ and $f \in {\rm Hol} (X)$.
By Lemma \ref{4.60} we have
$$F_X = {\rm Span} \{ \epsilon^{-1}  ({\rm d} (\omega_{f \delta})) \ : \ f \in {\rm Hol} (X), \delta \in \VFHO (X) \}$$ which in combination with the above description
of generators of  $H^{n-1} (X, \Omega^{n-1})$ yields the first statement.

To prove the second assertion note that the  $\Oeul (X)$-module homomorphism of sheafs
$$\Oeul (X)^N \to TX , \quad (f_1, f_2, \ldots , f_N) \mapsto  \sum_{i=1}^N f_i \delta_i$$ and, therefore (because of the isomorphism  $\nu \to \iota_\nu (\omega )$)
the homomorphism $$\Oeul (X)^N \to \Omega^{n-1}, \quad (f_1, f_2, \ldots , f_N) \mapsto \omega_{\sum_{i=1}^N f_i \delta_i}$$ are surjective.
Since $X$ is Stein, the Cartan Theorem B implies the surjectivity on the level of global sections.
Together with Lemma $\ref{4.60} $ this implies the second assertion.
\end{proof}

In combination with Lemma \ref{4.50} this implies the following.

\begin{corollary}\label{4.90} For a Stein manifold $X$ with the volume density property we have
finitely many vector fields $\delta_1, \delta_2, \ldots , \delta_N \in \IVFHO (X)$ for which
$$F_X = {\rm Span} \{ \delta_i  ({\rm Hol} (X)) |  \, i=1, \ldots, N \}. $$
\end{corollary}

\begin{lemma} \label{4.100}
If $X$ is Stein, then $F_X$ is a closed linear subspace of $\Hol (X)$ and thus a Frechet space. Moreover, $\Hol (X)  / F_X$
is isomorphic to $H^n (X, \C)$.
\end{lemma}

\begin{proof}

Since $X$ is Stein the inclusion of the holomorphic de Rham complex into the complex-valued de Rham  $C^\infty$-complex
$$\begin{matrix}\ldots & \ddd  & H^0(X, \Eeul^{n-1}) &\ddd & H^0 (X, \Eeul^n ) &\ddd& 0 \\
                           &           &      \cup             &         &   \cup            &       &     \\
 \ldots & \ddd &  H^0(X, \Omega^{n-1} ) &  \ddd  & H^0 (X, \Omega^n) & \ddd & 0 \\ \end{matrix} $$
is a quasi-isomorphism, i.e. the complex cohomology of $X$ can be computed via the lower complex.
By de Rham's theorem the boundaries ${\rm d} (H^0(X, \Eeul^{n-1} ))$ are exactly the
$n$-forms whose integral over any real $n$-dimensional cycle is zero (e.g., see \cite{GrHa})
and, therefore, ${\rm d } (H^0(X, \Omega^{n-1}))$ consists exactly of
the holomorphic $n$-forms which give zero when integrated over any real $n$-dimensional cycle. This condition is closed in the compact-open topology,
which proves the first assertion. The second one follows from $$H^n (X, \C) =H^0(X, \Eeul^n) / {\rm d} (H^0(X, \Eeul^{n-1})) \cong
H^0(X,\Omega^n) / {\rm d} (H^0(X, \Omega^{n-1})) \cong \Hol (X) / F_X,$$ where the last isomorphism is induced by $\epsilon$ from
lemma \ref{4.80}.
\end{proof}

\subsection{Grothendieck's tensor products}\label{4.104}
For convenience of readers we remind some facts from \textsc{Grothendieck}'s theory
of nuclear vector spaces which can be found in \cite{Sch54} or \cite{Tr}.
Denote by $E, E_i,F, F_i$, and $H$ locally convex
Hausdorff topological vector spaces.  \textsc{Grothendieck} introduced two different completions of the algebraic
tensor product $E_1 \otimes E_2$,
denoted by $E_1 \hat \otimes_\pi E_2$ and  $E_1 \hat \otimes_\epsilon E_2$. The crucial properties
of these completions we are going to exploit are the following. For linear continuous isomorphisms into (i.e embeddings)
$E_1 \to F_1$ and $E_2\to E_2$ the induced linear map $E_1\hat \otimes_\epsilon E_2 \to F_1\hat \otimes_\epsilon F_2$
is also an isomorphism into \cite[Proposition 43.7]{Tr}. In the case of metrizable $E_1$ and $E_2$
and linear continuous surjective maps $ E_1 \to F_1$ and $E_2\to E_2$ the induced
linear map $ E_1\hat \otimes_\pi E_2 \to F_1\hat \otimes_\pi F_2$
is also a surjection \cite[Proposition 43.9]{Tr}.
%Furthermore, one has
%$\Ker \varphi_1 \hat \otimes_\pi \varphi_2  = \Ker \varphi_1 \hat \otimes_\pi E_2  + E_1 \hat \otimes_\pi \Ker \varphi_2$ \cite[Exercise 43.2]{Tr}.

A space $E$ is called nuclear if for any other locally convex topological vector space
$F$ the completions $E \hat \otimes_\epsilon F$ and $E \hat \otimes_\pi F$ coincide
which allows us to omit indices   $\pi$ and $\epsilon$ in tensor products when appropriate.
Nuclear spaces possess nice properties: their subspaces and quotients with respect to closed subspaces
are again nuclear \cite[Proposition 50.1]{Tr}. Furthermore, the above claims about surjections and embeddings
in combination with \cite[Exercise 43.2]{Tr}\footnote{Using this exercise one can show
that for surjective maps $\varphi_1: E_1 \to F_1$ and $\varphi_2 : E_2\to E_2$ the induced
linear map $\varphi_1 \hat \otimes \varphi_2 : E_1\hat \otimes E_2 \to F_1\hat \otimes F_2$ has
$\Ker \varphi_1 \hat \otimes \varphi_2  = \Ker \varphi_1 \hat \otimes E_2  + E_1 \hat \otimes \Ker \varphi_2$.}
imply that the \textsc{Grothendieck}'s tensor product preserves
(short) exact sequences of metrizable nuclear spaces. That is,
$(E/H) \hat \otimes F = (E\hat \otimes F) /(H \hat \otimes F) $
where $H$ is a closed subspace of the nuclear space $E$.

\begin{example}\label{4.106} Let $X$ (resp. $Y$) be a closed Stein submanifold of $\C^n$ (resp. $\C^m$),
i.e we have the natural surjections $\varphi_X : \Hol (\C^n ) \to \Hol (X)$ and
$\varphi_Y : \Hol (\C^m ) \to \Hol (Y)$. The spaces of
holomorphic functions on Euclidean spaces are known to be nuclear \cite[Corollary, p. 530]{Tr} and, therefore, their
quotients $\Hol (X)$ and $\Hol (Y)$ are also nuclear. Furthermore $\Hol (\C^n) \hat \otimes \Hol (\C^m)$ is naturally
isomorphic to $\Hol (\C^{n+m})$  \cite[Theorem 51.6]{Tr} and the induced surjective linear map
$\varphi_X \hat \otimes \varphi_Y : \Hol (\C^{n+m}) \to \Hol (X) \hat \otimes \Hol (Y )$ has
the kernel generated by functions that vanish either on $X \times \C^m$ or on $\C^n\times Y$.%  \cite[Exercise 43.2]{Tr}.
Thus this kernel coincides with the defining ideal of $X\times Y$ in $\C^{n+m}$ and
$\Hol (X) \hat \otimes \Hol (Y )$ is naturally isomorphic to $\Hol (X \times Y)$.
\end{example}

We present the proof of the next simple fact because of the lack of references.

\begin{lemma}\label{4.108} Let $F_i$ be a closed subspace of a complete metrizable nuclear space $E_i$ for $i=1,2$.
Then $I:= (F_1 \hat \otimes E_2) \cap (E_1 \hat \otimes F_2)$ coincides with $F_1 \hat \otimes F_2$.

\end{lemma}

\begin{proof} Set $Q_i = E_i/F_i$. Then the kernel of the natural linear map $E_1 \hat \otimes E_2
\to Q_1 \hat \otimes Q_2$ coincides with $F_1 \hat \otimes E_2 + E_1 \hat \otimes F_2$.
 %\cite[Exercise 43.2]{Tr}.
 In particular, $Q_1 \hat \otimes Q_2$ is naturally isomorphic to the
 quotinet of $(E_1 \hat \otimes E_2) /(E_1 \hat \otimes F_2)= E_1 \hat \otimes Q_2$ with respect to
 $M:= (F_1\hat \otimes E_2)/I$ which implies that $M$ coincides with the subspace $F_1 \hat \otimes Q_2$
 of $E_1 \hat \otimes Q_2$ since $(E_1 \hat \otimes Q_2)/(F_1 \hat \otimes Q_2)= Q_1 \hat \otimes Q_2$.
 Note also that $F_1 \hat \otimes Q_2 \simeq (F_1 \hat \otimes E_2)/ (F_1 \hat \otimes F_2)$.
 Since $I \supset F_1 \hat \otimes F_2$ and $M:= (F_1\hat \otimes E_2)/I$ we must have $I=F_1 \hat \otimes F_2$.

\end{proof}

We need an extra product $E\epsilon F$ of locally convex topological vector spaces introduced by  \textsc{Schwartz}  \cite{Sch57}.
In general $E\hat \otimes_\epsilon F$ is a subspace of $E\epsilon F$. However we have equality
 $E\hat \otimes_\epsilon F=E\epsilon F$ provided that both $E$ and $F$ are complete and
 have the approximation property (see, \cite[Corollary 1, p. 47]{Sch57}).
 The latter means, say for $E$,  that the identical
operator on $E$ belongs to the closure of operators $E \to E$ of finite rank in the topology
$\cL_c (E, E)$ of uniform convergence on convex compact  subsets of $E$ (see, \cite[Definition, p. 5]{Sch57}).
For any nuclear space $E$ the identical operator can be always approximated
by operators of finite rank. Thus the result of  \textsc{Bungart}  \cite[Proposition 9.2]{Bu}
about $\epsilon$-product can be reformulated in the
special case of nuclear spaces as the following.

\begin{proposition}\label{4.110} Let $X$ be a Stein manifold and $E$ be a complete nuclear space.
Then $\Hol (X) \hat \otimes_\epsilon E$ coincides with the space of weakly holomorphic $E$-valued
functions on $X$ (i.e. $E$-valued functions $f$ on $X$ such that
if for any continuous linear functional $T$ on $E$ the function $T \circ f$ is holomorphic on $X$).

\end{proposition}

\begin{convention}\label{4.115} {\rm For Stein manifolds $X$ and $Y$ we consider holomorphic vector fields on $X\times Y$
tangent to the fibers of the natural projection $X\times Y \to X$ (resp. $X\times Y \to Y$) and we call them
vertical (resp. horizontal) fields.
Every holomorphic vector field $\delta$ and $X$ (resp. $\mu$ on $Y$)
generates the natural horizontal (resp. vertical) vector field on $X\times Y$ which by abuse of notation will
be denoted by the same symbol.}

\end{convention}

\begin{lemma} \label{4.120} Let Notation \ref{4.70} and Convention \ref{4.115} hold, and let $\delta_1, \delta_2, \ldots , \delta_N \in \IVFHO (X)$ be as in
Corollary \ref{4.90}. Then
for every $f\in F_X \hat\otimes \Hol (Y)$ there exist $\bar b = (b_1, \ldots , b_N) \in \Hol (X \times Y)^N$
such that $f= \delta_1 (b_1) + \delta_2 (b_2) + \ldots +
\delta_N (b_N)$ \footnote{In this formula each $\delta_j$ is already a horizontal field on $X\times Y$.}.
\end{lemma}

\begin{proof}
By Lemma \ref{4.100} the map $\Theta : \Hol (X)^N \to F_X$ given by $(a_1, a_2, \ldots, a_N) \mapsto \sum_{i=1}^N \delta_i ( a_i) $ is a linear
surjection of Frechet spaces.
By \textsc{ Grothendieck}'s theorem the linear map
$$\Theta \hat \otimes {\rm id}_Y :  (\Hol (X))^N \hat \otimes \Hol (Y) \cong (\Hol (X\times Y))^N \to F_X \hat \otimes \Hol (Y)$$
is surjective which implies the desired conclusion

\end{proof}

\begin{remark}\label{4.130} (1) The proof of existence of $\bar b$ from Lemma \ref{4.120} can be also extracted from
another nontrivial fact -
the theorem of  \textsc{Michael} (e.g., see \cite[p. 183-186]{Hol} and \cite[Corollary 17.67]{AB}) that states that
for every continuous surjective linear map of Frechet spaces there exists a homogeneous continuous section.

(2) The same argument implies that for every $f \in F_X \hat \otimes F_Y$ there is $\bar b
=(b_1, \ldots , b_N) \in (\Hol (X) \hat \otimes F_Y)^N$ for which $f = \delta_1 (b_1) + \ldots + \delta_N (b_N)$.

\end{remark}

\begin{proposition} \label{4.140}

For every vector field $\mu \in  \IVFHO (Y)$ and each $f \in F_X \hat\otimes \Hol (Y)$
there is a vector field $\alpha$ from the closure of $\Lie_{\rm hol}^{\omega_X \times \omega_Y}
(X \times Y)$ (in the compact-open topology) such that the field $f\mu - \alpha$ is a horizontal vector field on $X\times Y$.
\end{proposition}

\begin{proof}
By Lemma \ref{4.120} $f= \delta_1 (b_1) + \delta_2 (b_2) + \ldots +
\delta_N (b_N) $. Fix a closed embedding $\iota : Y \to \C^m_{y_1, \ldots , y_m}$ and denote
a monomial $y_1^{k_1} \cdots y_m^{k_m}$ by $y^k$, i.e. $k$ is the multi-index $(k_1, \ldots, k_m)$
with norm $|k| = k_1 + \ldots + k_m$. Then $b_i$ can be presented as
$ b_i = \sum_{k} b_{i,k} y^k$ where $b_{i,k} \in \Hol (X)$ and the sum
converges uniformly on compacts of $X \times Y$.
%As in the proof of the previous lemma (embedding $Y$ into $\C^m$) write
%
%$$f(x, y) = \sum _{k} a_k (x) y^k $$
%
%and find
%a section for each $a_k$, i.e., write $ a_k(x) = \sum_{i=1}^N  \delta_i (b_k^i (x))$ such that
%the sums $ b^i (x, y) = \sum_{k} b_k^i (x) y^k$ converge uniformly on compacts in $X\times \C^m$ (and thus on $X \times Y$).
Let us show that the desired limit can be given by the formula
$$\alpha = \lim_{M\to \infty}  \sum_{\vert k  \vert <M}  \sum_{i=1}^N [ y^k \delta_i\ ,\  b_{i, k} \mu ].$$
Indeed, since $y^k$ (resp. $b_{i,k}$) is in the kernel of $\delta_i$ (resp. $\mu$)
the involved vector fields are completely integrable and of ($\omega_X \times \omega_Y$)-divergence zero. Moreover
$$[ y^k \delta_i\ ,\  b_{i,k} \mu ] = y^k \delta_i (b_{i,k}) \mu - b_{i,k} \mu (y^k) \delta_i.$$
Thus $$ \sum_{\vert k  \vert <M}  \sum_{i=1}^N [ y^k \delta_i\ ,\  b_{i,k} \mu ] =
(\sum_{\vert k  \vert <M}  \sum_{i=1}^N   y^k \delta_i (b_{i,k})) \mu -  \sum_{i=1}^N  (\sum_{\vert k  \vert <M}
b_{i,k} \mu (y^k)) \delta_i =$$
$$ = (\sum_{i=1}^N \delta_i(\sum_{\vert k  \vert <M} b_{i,k} y^k )) \mu - \sum_{i=1}^N  \mu(\sum_{\vert k  \vert <M}    b_{i,k} y^k) \delta_i. $$

By the Weierstrass theorem about differentation of convergent power series we can send $M \to \infty$ and obtain
$$\alpha = f \mu - \sum_{i=1}^N  \mu(   b^i ) \delta_i$$ which yields the desired conclusion.

\end{proof}

\subsection{Proof of Theorem \ref{4.40}} Let completely integrable holomorphic vector fields $\mu_1, \ldots , \mu_M$ on $Y$ play the same role as $\delta_1, \ldots , \delta_N$ from Corollary \ref{4.90}
play for $X$. In particular, $\mu_1, \ldots , \mu_M$ generate the tangent vector space at any point of $Y$. Hence a simple application
of the Cartan theorem B implies that every vector field on $X\times Y$ is of form
$$\gamma = \sum_{i=1}^M f_i \mu_i  + \sum_{j=1}^N g_j  \delta_j $$ where $f_i$ and $g_j$ are holomorphic functions on $X\times Y$. We suppose further that $\gamma$ is of
$\omega$-divergence zero where $\omega =\omega_X \times \omega_Y$. Recall that for every holomorphic vector field
$\nu$ of $\omega$-divergence zero we have $\diver_\omega f \nu = \nu (f)$ for every holomorphic functions $f$.
Since $\mu_i$ and $\delta_j$ have $\omega$-divergence zero, the divergence of $\gamma$ coincides with
$$\diver_\omega \gamma = \sum_{i=1}^M \mu_i (f_i)  + \sum_{j=1}^N \delta_j (g_j) = 0. $$ Note that the first (resp. second)
summand can be viewed as a weakly holomorphic function on $X$ (resp. Y) with values in $F_Y$ (resp. $F_X$).

Thus by \textsc{Bungart}'s theorem (Proposition \ref{4.110}) we have $$f:= \sum_{i=1}^M \mu_i (f_i) \in \Hol (X) \hat \otimes F_Y
\, \, \, {\rm and} \, \, \, \sum_{j=1}^N \delta_j (g_j) \in F_X \hat \otimes \Hol (Y).$$ By Lemma \ref{4.108}
$\sum_{i=1}^M \mu_i (f_i) \in F_X \hat \otimes F_Y$. Furthermore, by Remark \ref{4.130} (2) there
exist $h_1, \ldots , h_M \in F_X \hat \otimes \Hol (Y)$ for which $$f =\mu_1 (h_1) + \ldots \mu_M (h_M).$$
By Proposition \ref{4.140} there exists a vector field $\nu \in \LieHO (X\times Y)$ of form
$$\nu = \sum_{i=1}^M h_i \mu_i  + \sum_{j=1}^N e_j  \delta_j. $$ Subtracting $\nu$ from $\gamma$ we
can suppose from the beginning that $$\diver_\omega  \sum_{i=1}^M f_i \mu_i = \sum_{i=1}^M \mu_i (f_i)=0$$
which implies that $$\diver_\omega  \sum_{i=1}^N g_i \delta_i = \sum_{i=1}^N \delta_i (g_i)=0$$ because
$\diver_\omega \gamma =0$. Since each $\mu_i$ is a vertical vector field on $X\times Y$ this means that the
restriction of $\sum_{i=1}^M f_i \mu_i$ to any fiber $x \times Y$ is of $\omega_Y$-divergence zero and therefore
it belongs to $\Lie_{\rm hol}^{\omega_Y} (x \times Y)$ by the assumption of the theorem. That is,
$\sum_{i=1}^M f_i \mu_i$ can be viewed as a weakly holomorphic function on $X$ with values in $\Lie_{\rm hol}^{\omega_Y} (Y)$.
By \textsc{Bungart}'s theorem (Proposition \ref{4.110}) it belongs to $\Hol (X) \hat \otimes \Lie_{\rm hol}^{\omega_Y} (Y)$
which is a subspace of $ \LieHO (X\times Y)$. The same argument implies that
$ \sum_{i=1}^N g_i \delta_i  \in  \LieHO (X\times Y)$ and thus $\gamma \in  \LieHO (X\times Y)$ which is
the desired conclusion.
\hspace{2.6in} $\square$

\section{Preliminary facts about foliation}

A more detailed exposition of the results from this section can be found in
\cite{Bru00},  \cite{Brun}, or \cite{Bru04}.

\begin{definition}\label{5.10}
(1) A foliation $\cF$ on a smooth complex surface $\bX$ is given by an open
covering $\{ U_j \}$ of $\bX$ and holomorphic vector fields $\nu_j \in H^0(U_j, T\bX )$
with isolated zeros such that
$$\nu_i = g_{ij} \nu_j \, \, \, {\rm on} \, \, \, U_i \cap U_j$$ for invertible
holomorphic functions $g_{ij} \in H^0 (U_i \cap U_j, \cO^*_X )$ where $\cO_X^*$
is the sheaf of invertible functions. Gluing orbits
of $\{ \nu_j \}$ one gets leaves of the foliation $\cF$. The singular set
${\rm Sing} \, (\cF )$ is the discrete subset of $\bX$ whose intersection with
each $U_j$ coincides with with zeros of $\nu_j$. The cocycle $\{ g_{ij} \}$ define
a holomorphic line bundle $K_{\cF}$ which is called the  canonical bundle
of the foliation $\cF$. %Its dual $T_{\cF}$ is called the tangent bundle of $\cF$

(2) This definition can be extended to the case of $\bX$ with quotient singularities only
where $\cF$ defined as a foliation on $\bX \setminus {\rm Sing} \, (\bX  )$.
%Following Brunella
We require additionally that
$${\rm Sing} \, (\bX ) \cap  {\rm Sing} \, (\cF ) = \emptyset .$$
That is, a singular point  $p$ of $\bX$ is locally of form $\B^2 /\Z_k$ where $\B^2$ is a ball in $\C^2$
equipped with a linear $\Z_k$-action.
In particular, the foliation can be lifted to $\B^2\setminus \{ (0,0) \}$ and the requirement is
that it can extended to a foliation on $\B^2$ with a non-vanishing associated vector field $\nu$
(and this must be true for any singular point of $\bX$).

Then $K_{\cF}$ on $\bX$
is the direct image of the canonical bundle on $\bX \setminus {\rm Sing} \, (\bX )$
under the inclusion morphism $\bX \setminus {\rm Sing} \, (\bX ) \hookrightarrow \bX$
(in this situation $K_\cF$ is not a bundle but only a $\Q$-bundle). % More precisely,
%the vector field $\nu$ on $\B^2$ mentioned above is not $\Z_k$-invariant but
%its power $\nu^{\otimes k}$ is.

(3) Foliation $\cF$ is called nef if $K_{\cF}$ is nef.

(4) A singularity $p \in {\rm Sing} \, (\cF )$ is reduced if the linear part
of the corresponding vector field at $p$ has eigenvalues $\lambda_1, \lambda_2$
such that either they are nonzero and $\lambda_1/\lambda_2 \notin \Q_+$ or
$\lambda_1 \ne 0 =\lambda_2$.  %In the former case $p$ is called nondegenerate,
%in the latter case a saddle-node.
The foliation $\cF$ is called reduced if
all of its singularities are reduced.

(5) The Kodaira dimension ${\rm kod} (\cF )$ of a reduced foliation $\cF$
on a projective surface $\bX$ is the Kodaira-Iitaka dimension of
its canonical bundle $K_{\cF}\in {\rm Pic} (\bX) \otimes \Q$. That is,

$${\rm kod} (\cF ) = \lim{\rm sup}_{n \to + \infty} {\frac{\log \dim H^0 (\bX, K_{\cF}^{\otimes n})}{\log n }}.$$

\end{definition}

We shall study foliations associated with
completely integrable holomorphic vector fields on a Stein surface $X$ and
the following result of  \textsc{Suzuki} \cite{Su77a}, \cite{Su77b} is very important.

\begin{theorem}\label{5.20}  Let $\cF$ be a foliation on a normal Stein surface $X$.

{\rm (1)} If all leaves of $\cF$ are properly embedded in $ X \setminus {\rm Sing} \, (\cF  ) $
then there is a nonconstant meromorphic first integral of $\cF$ on $X$. % and a holomorphic map
%$\rho : X \to R$ into a Riemann surface $R$ such that
%
%{\rm (1a)} The irreducible components of the fibers of $\rho$ are the leaves of $\cF$.
%
%{\rm (1b)} The union $E \subset X$ of all reducible fibers of $\rho$ has zero logarithmic capacity.

{\rm (2)} Furthermore, if the general leaf of $\cF$ is isomorphic to $\C^*$ (we shall
call below such foliations of $\C^*$-type) then every leaf is
closed in $X \setminus {\rm Sing} \, (\cF  )$ and therefore there is a meromorphic first integral as in (1).
\end{theorem}

For such a normal affine
algebraic surface $X$ the study of the foliation $\cF$
would be much simpler when this first integral were a rational function.
We need to consider the situation when this first integral is not rational.

First note that a foliation associated with an algebraic vector fields on $X$
can be extended to a completion $\bX$ of $X$ and we can deal with a projective normal
surface. Then we have the following result of \textsc{Seidenberg} (e.g., see \cite{Bru00}).

\begin{theorem}\label{5.30}
Performing a sequence of blow-ups $\tX \to \bX$ we obtain
a lifted foliation $\tilde \cF$ on $\tilde X$ which is reduced.

\end{theorem}

Thus from now on we shall work with reduced foliations on projective surfaces.
 \textsc{Miyaoka} and  \textsc{Shepherd-Barron} (e.g., see \cite{Brun}) established the following.

\begin{theorem}\label{5.40}  If $\cF$ is a reduced foliation on a projective surface $\bX$ with
at most cyclic quotient singularities then $K_\cF$ is pseudoeffective if and only
if $\cF$ is not a rational foliation (i.e. its general fiber is a rational curve in $\C \proj^1$).

\end{theorem}

The Zariski-Fujita decomposition implies that every pseudoeffective $\Q$-bundle is
the sum of a nef $\Q$-bundle and a negative part for which the associated
divisor can be contracted. This is a basis for the next fact (\textsc{McQuillan}'s contraction) (e.g., see \cite{Bru00} or \cite{Brun}).

\begin{theorem}\label{5.50}
Let $\cF$ be a non-rational reduced foliation on a projective surface $\bX$ with
at most cyclic quotient singularities. Then there exists a birational morphism
$(\bX , \cF ) \to (\bX' , \cF')$ such that $\bX'$ is still projective with at most cyclic
quotient singularities, $\cF'$ is still reduced, and $K_{\cF'}$ is nef.

\end{theorem}

\begin{remark}\label{5.60} (\cite[Section 3]{Bru04} or \cite[page 10]{Brun})
Contraction of $\bX$ to $\bX'$ is a sequence
of blowing down of rational curves such that each of them is invariant
with respect to the consequent induced foliation and the restriction of the canonical
bundle of the foliation to the curve is negative. Every of these curves $F$ contains exactly one singularity
$p$ of the foliation which is automatically a regular point of the surface. Furthermore, $F$ is contracted
to a point which is a regular point of the induced foliation on the resulting surface (but not in general a regular
point of the surface).
%and it  may also contain one cyclic quotient singularity
%$q \ne p$ of order $m$. The singularity type of $p$ is ${\rm d } (z^nw^m)$ with
%$F = \{ w=0 \}$ and $m=1$ in the case of absence of $q$. After contraction one gets
%a new cyclic quotient singularity of order $n$ (which in any case a regular point of the foliation).

\end{remark}

The reduced foliation generated by a completely integrable algebraic vector field on $X$
admits a lot of entire tangent curves $\C \to \bX$.
If $\cF$ has no rational first integral, then such general curve is Zariski dense in $\bX$ (Darboux's
theorem) and another result of \textsc{McQuillan} says the following (see \cite[Sections IV and  V]{McQ} or
\cite[Chapter 9, Theorems 1 and 4, Corollary 1]{Bru00}.

\begin{theorem}\label{5.70} Let $\cF$ be a reduced foliation on a smooth projective
surface $\bX$ such that $\cF$  possesses a tangent nonconstant entire curve
that is Zariksi dense in $\bX$. Then the Kodaira dimension ${\rm kod} (\cF )$ is either 0 or 1.

{\rm (1)} Furthermore, if  ${\rm kod} (\cF )=1$ then either

{\rm (1a)} $\cF$ is a Riccati foliation, i.e. there exists a fibration $f: \bX \to B$ whose general
fiber is a rational curve transverse to $\cF$ or

{\rm (1b)} $\cF$ is a Turbulent foliation, i.e. there exists a fibration $f: \bX \to B$ whose general
fiber is an elliptic curve transverse to $\cF$.

{\rm (2)} If  ${\rm kod} (\cF )=0$ and $\bX \to \bX'$ is the \textsc{ McQuillan}'s contraction to a
nef reduced foliation $\cF'$ on $\bX'$ then
%$\cF'$ is also numerically trivial, i.e.
%$K_{\cF'} \cdot C =0$ for every irreducible curve $C \subset \bX'$. Moreover,
there exists
a finite covering $r: Y \to \bX'$ such that

{\rm (2a)} $Y$ is smooth and $r$ is ramified only over the quotient singularities of $\bX'$.
%(in fact $r$ is unramified if we give $\bX'$ its natural orbifold structure).

{\rm (2b)} The canonical bundle $K_\cG$ of the lifted foliation $\cG = r^* (\cF' )$ is trivial,
i.e. $K_\cG = \cO_Y$, and so $\cG$ is generated by a global holomorphic vector field
with isolated zeros only.
%\footnote{Since {\em loc.cit} \cite{McQ} and \cite{Bru00} may not be
%easily available, let us explain briefly how to extract the statements (1) and (2) of Theorem \ref{5.70} from
%\cite{Brun} where besides ${\rm kod} (\cF )$ Brunella considers so-called numerical
%Kodaira dimension $\nu (\cF )$.  The second statement is actually established for $\nu (\cF)=0$
%\cite[Theorem 2]{Brun} but  \cite[Theorem 3]{Brun} implies that conditions $\nu (\cF )=0$ and
%${\rm kod} (\cF )=0$ are equivalent. When ${\nu } (\cF )=1$ one has either
%${\rm kod}  (\cF )=1$ or  ${\rm kod}  (\cF ) = \infty$ \cite[Theorem 3]{Brun} and the latter case can be
%excludes because otherwise $\bX$ must be Kobayashi hyperbolic  \cite[Theorem 7]{Brun}  and cannot contain
%entire curves.  The first statement of Theorem \ref{5.70} was established for $\nu (\cF ) = {\rm kod} (\cF)=1$
% \cite[Theorem 4]{Brun} with two other possibilities (besides the Riccati and Turbulent foliations) which have to
% be disregarded since their general leaves are curves of positive genus. }

\end{theorem}

\section{\textsc{Brunella}'s construction: the case of  ${\rm kod} (\cF )=1$. }

In fact  \textsc{Brunella} proved more (Lemmas 1 and 2 in \cite{Bru04}).
%established that in the absence of rational first intergral
%(1b) in Theorem \ref{???}  cannot hold (see, Lemma 1 in \cite{Bru04}),
%i.e. when  ${\rm kod} (\cF )=1$
%for a reduced foliation without a rational first integral then it is a Riccati foliation
%and we proceed with his results.
%He proved the following (\cite{Bru04}, Lemma 2).

\begin{proposition}\label{6.10}
{\rm (1)} Case
(1b) in Theorem \ref{5.70}  cannot hold,
i.e. when  ${\rm kod} (\cF )=1$
for a reduced foliation without a rational first integral then it is a Riccati foliation

{\rm (2)} Contracting curves in the fibers of $f: \bX \to B$ for the Riccati foliation
on a smooth surface $\bX$
one can suppose that each fiber of $f$ belongs to one of five standard types and
there is always at least one singular fiber of $f$
of so-called types (e) or (d) (see \cite[Lemma 2]{Bru04}) that consists of a union of
leaves of the foliation $\cF$
%the fibers of $f$ belong to one of five types (a), (b), (c), (d), (e)
%where (a) corresponds to general fibers and the rest consist of singular fibers. Furthermore,
%fibers of type (c), (d), and (e) are the union of leaves of the foliation $\cF$ and
%there is always at least one fiber of type (d) or (e)
(it will be denoted
below by $F_\infty$).

\end{proposition}

Next we have the following (Lemmas 3-5 in \cite{Bru04}).

\begin{proposition}\label{6.20} Let $\nu$ be a completely integrable algebraic vector field on
a rational affine algebraic surface $X$ and $\pi : \hX \to X$ be the resolution of
the singularities of $\nu$ (i.e. the lift $\hat \nu$ of $\nu$ is a completely
integrable vector field on $\hX$).  Suppose that the foliation $\cF$ generated
by $\hat \nu$ on the completion $\bX$ of $\hX$ is reduced and has no rational first integral.
Let $f: \bX \to B$ be the corresponding (rational) Riccati fibration (note that $B \simeq
\C \proj^1$ since $\bX$ is rational). Then

{\rm (1)}  the vector field $\hat \nu$ preserves $f|_{\hX}$ (i.e. its flow
sends fibers into fibers) or in other words $\hat \nu$ is a lift of
a completely integrable vector field $\nu_0$ on $f (\hX ) \subset B$;

{\rm (2)} the fiber  $F_\infty$ is contained in the divisor $D = \bX \setminus \hX$
and in particular the exceptional divisor $E$ of $\pi$ is
disjoint from $F_\infty$;

{\rm (3)} contracting $E$ we get from $f$ a regular function $P$ on $X$ whose general
fibers are either $\C$ or $\C^*$ (i.e. again the  flow of $\nu$ maps each fiber of $P$ onto a fiber of $P$).

\end{proposition}

In fact for foliations of $\C^*$-type  the same conclusion can be made even without assumption
that ${\rm kod} (\cF )=1$. In order to demonstrate it we need some facts from another
 \textsc{Brunella}'s paper \cite{Bru98} where he used a slightly different terminology.
Let $L$ be a leaf of the foliation $\cF$ on $X$ and
$L_0$ be a Riemann surface isomorphic to $\{ z \in \C | \, 0\leq r < |z| \leq 1 \}$
which is properly embedded into $L$. Then $L_0$ is called a planar isolated end of $L$.
This end is called transcendental if the set $\bar L_0 \setminus L_0$
consists of more than one point where $\bar L_0$ is the closure of $L_0$ in $\bX$.
Fibration $\cF$ is called (in \cite{Bru98})  $P$-complete
for some regular function $P$ on $X$ if there exists a finite set $Q$ such that for every
$t \in \C \setminus Q$ the fiber $P^{-1} (t)$ is transversal to $\cF$ and there is
a neighborhood $U_t\subset \C$ of $t$ for which $P|_{P^{-1} (U_t)}$ is a fibration
and $\cF |_{P^{-1}(U_t)}$ define a local trivialization of it.

The next fact is the main result in \cite{Bru98}.

\begin{proposition}\label{6.30} Let $X$ be a smooth rational affine surface with an SNC-completion
$\bX$ and $\bar D=\bX \setminus X$. Suppose that for any other SNC-completion $\bar \bX$ of
$X$ that dominates $\bX$ (i.e. there is a morphism $\pi : \bar \bX \to \bX$ identical
on $X$) the following is true.

{\rm (i)} There is a Kahler metric on $\bar \bX$ such that the restriction of the associated
2-form to $X$ is exact.

{\rm (ii)} Let $\Gamma$ be the dual graph of the divisor $\bar D$ and let $\Gamma$ have no
linear $(-1)$-vertices. Suppose also that  for any branch point $b$
of $\Gamma$ such that the subgraph $\Gamma \ominus b$ contains two connected components (among others),
each of which is contractible to a $(-2)$-vertex, the weight of $b$
is at most $-2$ (note that this condition
will hold automatically for the dual graph of
$\bar {\bar D}= \bar \bX \setminus X$).

Then for any foliation $\cF$ generated by a regular vector field on $X$, which possesses also a
transcendental planar isolated end,
there exists a regular function $P\in \C [X]$
with general fibers isomorphic  either to $\C$ or to $\C^*$ such
that $\cF$ is $P$-complete.
\end{proposition}

\begin{remark}\label{6.40} (1) This regular function $P$ yields, of course, a Riccati foliation which is mentioned by  \textsc{Brunella}
\cite[p. 1241]{Bru98}. He mentions also in \cite{Bru04} that the construction from
\cite[pp. 1241-1243]{Bru98} implies that one of singular fibers of $P$ is of type (d).

(2) Suppose that $\cF$ is a foliation of $\C^*$-type that has no rational first integral. Then general leaves
are isomorphic to $\C^*$, they are Zariski dense in $\bX$ and properly embedded into the complement to singularities
by Theorem \ref{5.20}. This implies the existence of a transcendental planar isolated end, i.e. Proposition \ref{6.30}
is applicable to foliations of $\C^*$-type generated by regular vector fields.

(3) In \cite{Bru98} this Proposition \ref{6.30} is proven
only for $X=\C^2$. However the analysis of the proof shows that it is valid
for any smooth rational affine surface $X$ satisfying condition (i) and (ii).
Condition (i) (which is obviously true for $\C^2$ since
the second cohomology of $\C^2$ is trivial) is used in Lemma 3 in \cite{Bru98}.
Condition (ii) is also true for $\C^2$ and used in the proof of Lemma 8 in \cite{Bru98}.
The next (certainly well-known) fact says more about condition (i).
\end{remark}

\begin{proposition}\label{6.50} Condition (i)  from Proposition \ref{6.30}
is automatic for every affine algebraic manifold $X$.

\end{proposition}

\begin{proof}  We can choose an SNC-completion $\bX$ of $X$ so that $\bX \setminus X$ is a support
of an ample divisor (for a two-dimensional $X$ the facts that $X$ is affine and that  $\bX \setminus X$ is the
support of an ample divisor are equivalent). Then this divisor generates an embedding $\bX \hookrightarrow \proj^n =:\bY$
such that for a hyperplane $H\subset \bY$ the Euclidean space $\bY\setminus H =:Y\simeq  \C^n$ contains $X$ as a closed submanifold.
The birational morphism $\pi : \bar \bX \to \bX$
may be viewed as the blow-up of an ideal sheaf $\cI$ \cite[Theorem II.7.17]{Har}. Consider the ideal sheaf on $\bY$
generated by $\cI$ and the defining equations of $\bX$. The blow-up with respect
to this new sheaf yields a birational morphism $\tau : \bar \bY \to \bY$ which
is identical over $Y$, i.e. $\bar \bY$ is another completion of $Y$. Furthermore, $\pi$ is
the restriction of $\tau$ to the proper transform of $\bX$ \cite[Corollary II.7.15]{Har}.
Since $\bar \bY$ is projective it is a subvariety of $\proj^N$ which
is equipped with a closed 2-form $\omega$ as a Kahler manifold. The restriction of $\omega$ to $Y$ is an
exact form by de Rham's theorem because
the second cohomology of $Y$ is trivial. Thus its restriction to $X$ is also
exact which is the desired conclusion.

\end{proof}

\begin{proposition}\label{6.60}  (cf.  \cite[Proposition 3]{Bru04}) Let $\nu$ be a completely
integrable algebraic vector field of type $\C^*$
on a smooth affine rational surface $X$ satisfying condition (ii)
from Proposition \ref{6.30}. Then either $\nu$ possesses
a rational first integral or for some
regular function $P$ on $X$  with
general fibers $\C$ or $\C^*$ the flow of $\nu$ transforms fibers of $P$ into
fibers of $P$.

\end{proposition}

\begin{proof} By Theorem \ref{5.20} the leaves of the associated foliation $\cF$ are
properly embedded into $X \setminus {\rm Sing} (\cF )$ and therefore there is a meromorphic
first integral. Assume that it is not rational. Then there is at least one transcendental
end in $X$ (otherwise all leaves are algebraic and
the first integral is rational by the Darboux's theorem). Now Proposition \ref{6.30}
and Remark \ref{6.40} imply the existence of a Riccati fibration $P$
with general fibers $\C$ or $\C^*$ such that
the fiber over $\infty$ is of type (d) which is the only thing needed
for validity of Proposition \ref{6.20} whose proof in \cite{Bru04} does not use other
properties of foliations with  ${\rm kod} (\cF )=1$.

\end{proof}

\section{\textsc{Brunella}'s construction: the case of ${\rm kod} (\cF )=0$.}

Thus we can consider only completely integrable algebraic fields of type $\C$
for which most of orbits are isomorphic to $\C$.  For a smooth affine surface $X$ with
 a completely integrable algebraic
 vector field $\nu$ we consider its smooth completion
 $\tX$ and the resolution of singularities $\pi : \bX \to \tX$ of the foliation
 induced by the vector field so that the resulting foliation $\cF$ on $\bX$ is reduced.
 Then $D = \pi^{-1} (\tX \setminus X )$ is the divisor at infinity and $E$ is the
 exceptional divisor of $\pi$ over $X$, i.e $\pi (E)=\Gamma$ is a finite subset
 of $X$.
Moreover we have the following \cite[p. 442]{Bru04}.

\begin{lemma}\label{7.10}  Each irreducible component of $D$ and each
irreducible component of the exceptional divisor $E$ of $\pi : \bX \to \tX$
are $\cF$-invariant.
\end{lemma}

Following  \textsc{Brunella}, consider now the \textsc{McQuillan}'s contraction $\tau : \bX \to \bX'$
of a reduced foliation $\cF$ on $\bX$ to a nef reduced foliation $\cF'$ on $\bX'$
as a sequence of blowing downs
$$\bX=\bX_0 \to \bX_1 \to \ldots \to \bX_{n-1} \to \bX_n=\bX'$$
 of rational curves $F_j\subset \bX_j$ as in Remark \ref{5.60}.
 Denote by $\cF_j$ the foliation induced by $\cF$ on $\bX_j$, by $D_j$ and $E_j$
 the images of $D$ and $E$ respectively. We call $F_j$ external if
 it is not contained in $D_j \cup E_j$.

 \begin{lemma}\label{7.20}  {\rm (cf. \cite[ Lemma 6]{Bru04})} Let
 $F_j$ be external and let
 $\tF_j$ be the proper transform of $F_j$ in $\tX$.
 Then $\tF_j \cap X$ is an algebraic curve in $X \setminus \Gamma$ isomorphic to $\C$. Furthermore, for different
 external curves $F_j$ and $F_i$ the curves $\tF_j \cap X$ and $\tF_i \cap X$ are disjoint.

\end{lemma}

\begin{proof} Note that $F_j$ must meet $D_j$ since otherwise its strict
transform in $\bX$ will be contained in $\pi^{-1} (X)$ but not in $E$ which
is impossible for a complete curve. Then $D_j\cap F_j$ must be $\cF_j$-invariant
since both $D_j$ and $F_j$ are. Therefore, it is the only singular point of the
foliation on $F_j$ (see Remark \ref{5.60})). By the same reason $F_j$ cannot meet $E_j$
which implies the first statement.
To see the second one we
note that $F_j$ and $F_i$ cannot meet outside $D$
since otherwise contrary to  Remark \ref{5.60}
$F_j$ has another singular point of the foliation besides $D_j \cap F_j$.
Thus $\tF_j\cap X$ and $\tF_i \cap X$ are disjoint.
\end{proof}

\begin{remark}\label{7.30} Let  $L$ be the union  $\bigcup \tF_j$  where $F_j$ runs over the set of external curves.
By construction $L$ and therefore $X^*=X \setminus (L \cup \Gamma )$ are invariant
with respect to the original foliation, i.e. the restriction of $\nu$ to $X^*$ is completely
integrable.

\end{remark}

Then  \textsc{Brunella} considers morphism $r : Y \to \bX'$ as in Theorem \ref{5.70} and the commutative
diagram generated by it
\[  \begin{array}{ccc} Z &  \stackrel{{ h}}{\rightarrow} & \bX \\
\, \, \, \, \downarrow {g}  & & \, \, \, \, \downarrow { \tau} \\
Y &  \stackrel{{ r}}{\rightarrow} & \bX'

\end{array} \]
where $g: Z \to Y$ is a birational morphism and $h: Z \to \bX$ is a ramified covering
such that $r \circ g = \tau \circ h$. Set
$R = D \cup E$,
$R' =\tau^{-1} (\tau (R))$, and
 $W= Z \setminus h^{-1} (R' )$.
Suppose that $\mu_0$ is the holomorphic vector field on projective variety $Y$
that generates fibration $\cG$ from Theorem \ref{5.70}. Let $\mu$ be its
lift to $Z$ via $g$.

\begin{lemma}\label{7.40}  {\rm (cf. \cite[Lemma 7]{Bru04})} Let $\mu , W, X^*, L$ be as before.
Then

{\rm (1)} the restriction of $h$ makes $W$ an unramified cover of $X^*$;

{\rm (2)} the restriction $g|_W :  W \to Y$ is an embedding;
%{\rm (2)} the restriction of $\mu$ is a holomorphic completely integrable vector
%field on $W$ which has poles along $V\setminus W = h^{-1} (L)$;

{\rm (3)} elements of the flow associated with the restriction of $\mu |_W$ are
algebraic automorphisms of $W$.

\end{lemma}

\begin{proof} By Theorem \ref{5.70} morphism $r$ can be ramified only
over singularities of $\bX'$. By construction these singularities are contained in $\tau (R)$. Thus $h$ can be ramified only over
$R' $. For (1) it remains to note that $R' =D \cup E \cup L$.

Since $Z$ is the fibered product of $Y$ and $\bX$ over $\bX'$ we see that $W$ is the fibered product
of $Y$ and $\bX \setminus R'$ over $\bX'$. Hence for (2) it suffices to note that the restriction of $\tau$ to $\bX \setminus R'$
is an embedding.

%The curve $r^{-1} (\tau (R))$ is $\cG$-invariant and therefore $\mu_0$ restricted to
%its complement is still complete, i.e. the first part of (2) follows from the fact that this complement
%and $W$ are isomorphic via $g$. The second part follows from the fact that $g$ is
%a composition of blow-ups over regular points of $\cG$ and therefore is totally
%polar over the exceptional divisor of $g$.

We have also (3) since $\mu$ arises from a holomorphic vector field $\mu_0$
on the projective variety $Y$ and elements of its flow are algebraic automorphisms
by Chow's theorem.

\end{proof}

\begin{corollary}\label{7.50}  Suppose that for any finite collection of disjoint lines in $X \setminus \Gamma$
and every finite unramified covering $W \to X \setminus (\Gamma \cup L)$ where $L$ is the
union of these lines the manifold $W$ has a finite number of algebraic automorphisms. Then
the case of ${\rm kod } (\cF )=0$ does not hold.

\end{corollary}

\section{Geometry of the surface $x+y+xyz=1$.}

Consider the hypersurface $S\subset \C^3_{x,y,z}$ given by   $x+y+xyz=1$.
It can be viewed as the affine modification of $\pi : S\to \C^2$ of
$\C^2_{x,y}$ along the divisor $xy=0$ with
center at $xy=1-x-y=0$ (i.e at points $(1,0)$ and $(0,1)$) since $z =(1-x-y)/(xy)$.
Note that $S$ contains four lines $L_1,L_2,L_3$, and $L_4$ such that $\pi (L_1)= (0,1)$,
$\pi (L_2)= \{ y=1 \}$, $\pi (L_3)= (1,0)$, and $\pi (L_4) = \{x =1 \}$.

\begin{lemma}\label{8.10} There is only one polynomial curve $L_5$ in
$S$ different from $L_1, L_2,L_3,L_4$ and it is the proper transform of the
line $x+y=1$ in $\C^2$.

\end{lemma}

\begin{proof} Consider a non-constant morphism $\phi : \C \to S$ whose image
is not one of the four lines. Its composition
with $\pi$ yields a morphism $\C \to \C^2$ given by $x=p (t)$ and $y=q(t)$
(where $p$ and $q$ are non-constant polynomials)
whose image $C$ may
meet the $x$-axis and the $y$-axis only at points $(1,0)$ and $(0,1)$ respectively.
Furthermore, local branches of $C$ cannot be tangent to any of these axes
since otherwise the proper transform of $C$ in $S$ (which is $\phi (\C )$)
is a Riemann surfaces with at least two punctures contrary to the assumption that
it is a polynomial curve. This implies that $1-p(t)$ is divisible by $q(t)$ while
$1-q(t)$ is divisible by $p (t)$. In particular, $\deg p = \deg q$ and therefore,
$p (t)= 1-q(t)$. This yields the desired conclusion.

\end{proof}

\begin{remark}\label{8.20} One can check that the equations of $L_1, L_2,L_3,L_4$,
and $L_5$ in $S$ are $x=0$, $yz+1=0$, $y=0$, $xz+1=0$, and $z=0$ respectively.

\end{remark}

\begin{lemma}\label{8.30} Let $\rho : S' \to S$ be a finite morphism of smooth
surfaces unramified over the complement to the five lines in $S$.
Then $S'$ has a finite number of algebraic automorphisms.
\end{lemma}

\begin{proof}  First note that for every polynomial curve $C$ in $S'$ its
image $\rho (C)$ is a polynomial curve in $S$. Thus the number of such curves
in $S'$ is finite. Any automorphism $\alpha$ of $S'$ generates their permutation.
Without loss of generality we can assume that this permutation is
identical. Then the restriction of $\alpha$ yields an automorphism of the
surface $H' \subset S'$ that is an unramified finite covering of the Kobayashi
hyperbolic surface
$H=S \setminus \bigcup_{i=1}^5 L_i$ (and therefore $H'$ is itself
hyperbolic). There are at most finite number
of such automorphisms and we have  the desired conclusion.

\end{proof}

\begin{proposition}\label{8.40} Let $\nu$ be a completely integrable algebraic vector field
on $S$.  Then

{\rm (1)} either $\nu$ has a first integral which is a regular function with general fibers
isomorphic to $\C^*$ or

{\rm (2)} there is a regular function $P$ with general fibers isomorphic to $\C^*$
that is a Riccati fibration for $\nu$, i.e. the  flow associated with $\nu$ transforms
fibers of $P$ into fibers of $P$.

\end{proposition}

\begin{proof}  We need to check conditions (i) and (ii) in Proposition \ref{6.30}.
The first of them is always true because of Proposition \ref{6.50}. As for the second,
the surface $S$ possesses an SNC-completion whose dual graph is
a cycle consisting of three vertices of weights $1,0$, and $0$ respectively.
Thus Proposition \ref{6.30} is applicable.

Assume now that $f$ is a rational first integral of $\nu$. Then general orbits of $\nu$
are contained in the fibers of $f$ on the complement $S^*$ to the set of indeterminacy
points of $f$ in $S$. Since $S$ contains only five polynomial curves these fibers
must be isomorphic to $\C^*$. Furthermore, since their closures must
meet at indeterminacy points we see that such points cannot exist because
otherwise the closures produce infinite number of polynomial curves. Thus $S^*=S$ and
$f$ is regular which is (1).

In the absence of a rational first integral for $\nu$ let $\cF$ be the
associated foliation on an SNC-completion of $S$.
By Corollary \ref{7.50}  and Lemma \ref{8.30} ${\rm kod} (\cF )$ cannot vanish,
i.e. it is equal to 1 by Theorem \ref{5.70}. Now Proposition \ref{6.20} implies (2).

\end{proof}

\begin{lemma}\label{8.50} The surface $S$ is factorial and
its Euler characteristics is 2.

\end{lemma}

\begin{proof} Note that $S = T \cup L_1\cup L_3$ where $T = S \setminus (L_1 \cup L_3)$
is a torus.
Hence $\chi (S ) = \chi (T) + \chi (L_1)+ \chi (L_3)=2$ by \cite{Dur}. Then
modification $\pi : S \to \C^2$ may be viewed as composition of two
modifications $\pi' : S \to S'$ and $\pi'' : S' \to \C^2$ where $S'$ is the affine modification
of $\C^2$ along $x=0$ at center $(0,1)$ (i.e. $S'$ is again isomorphic to $\C^2$
and therefore factorial) and $S$ is obtained as the affine modification along the proper
transform of $y=0$ in $S'$ (which is a $\C^*$-curve) with center at point $(1,0)$.
That is, $S$ is obtained from $S'$ by replacing an irreducible curve by a line $L_3$.
The Nagata lemma (e.g., see \cite[Lemma 19.20]{Eis}) implies that $S$ remains factorial.

\end{proof}

\begin{proposition}\label{8.60}

Let $f$ be a regular
function on $S$ whose general fibers are isomorphic to $\C^*$.
Then up to a linear transformation $f$ coincides with one of the functions $x^ky^l$
(where $k$ and $l$ are relatively prime
nonnegative integers with $k+l \geq 1$), $xz+1$, $yz+1$, and $z$.

\end{proposition}

\begin{proof}  Since the general fibers are isomorphic to $\C^*$ each connected component
of singular fibers is either $\C^*$ or a curve of Euler characteristics one (more precisely either
a line or a cross) by \cite[Theorem 3]{Zai}.  Furthermore, there are exactly two such components
because the Euler characteristics of $S$ is 2 by Lemma \ref{8.50}. They are disjoint
and therefore the description of polynomials curves in $S$ implies that the union of these
components contains one of the following pairs of lines: $(L_1, L_3),  (L_1, L_4),  (L_2,L_3),
(L_2, L_5)$, and $(L_4,L_5)$.  Consider the case of the first three pairs and tori
 $T_{13}=S \setminus  (L_1 \cup L_3)$,  $T_{14}=S \setminus  (L_1 \cup L_4)$,
and  $T_{23}=S \setminus  (L_2 \cup L_3)$.  The general fiber $F\simeq \C^*_\xi$
is contained in one of these tori,
say $T_{13}$ that is isomorphic to $\C^*_x\times \C^*_y$. Hence up
to constant factors $x=\xi^k$ and $y=\xi^l$ where $k$ and $l$ are coprime integers
since the functions $x$ and $y$ separate points of $F$. Let $m$ and $n$ be integers
such that $mk+nl=1$. Then $\xi =x^my^n$ together with $f= x^ky^{-l}$ form a new coordinate
system on $T$ such that the restriction of $f$ to $F$ is constant. Note that this $f$ is regular
on $S$ only when the powers are nonnegative which yields the first functions mentioned
in the formulation.

Note that $T_{14}$ is isomorphic to the torus
$\C^*_x \times \C^*_\zeta$ where $\zeta =y/(x-1)$. Thus for $T_{14}$ this function
$f$ must be of form $x^k(x-1)^l/y^l$. However for a nonzero $l$ this function
is not regular on $S$ and we have to disregard it.  Similarly in the case
of  $T_{23}$ we have to disregard function of the form $y^k(x-1)^l/x^l$.

Now consider the case when the zero fiber of $f$ contains $L_4$ and $L_5$ as components.
Let us show that
$f^{-1} (0)=L_4 \cup L_5$. Indeed, assume to the contrary that this fiber contains another irreducible
component $L$. Since $S$ is factorial by Lemma \ref{8.50}  $L_4$ is given by zeros of $xz+1$,
$L_5$ by zeros of $z$ (see Remark \ref{8.20}), and $L$ by zeros of another function $h$.
All three functions are invertible on  $T_{45} =S \setminus f^{-1}(0)$ and thus the group of
invertible functions has at least three generators. This group is isomorphic to the group of integer
first cohomology (e.g., see \cite{Fuj}) and therefore $H^1(T_{45}, \Q )$ is at least of rank 3.
On the other hand nonzero fibers of $f$ do not contain
connected components of Euler characteristics one and
therefore $T_{45} $ is a $\C^*$-fibration over $\C^*$. By the technical
Lemma \ref{8.70} below its rational cohomology is at most of rank 2. This contradiction
shows that $L$ cannot exist.

Since $f^{-1}(0) = L_4\cup L_5$ up to a constant factor we have
$f= z^k (xz+1)^l$ where $k$ and $l$ are natural and, furthermore, relatively
prime (since otherwise the general fiber of $f$ is not connected). Then taking into consideration the equation of $S$
we see that the equation of the projection $C$ of the fiber $f^{-1} (c)$ to $\C^2_{x,y}$ is given
by $$(1-x-y)^k(1-x)^l=c x^ky^{k+l}. \eqno{(8.1)} $$
Hence over a neighborhood of $x=0$ (resp. $x=\infty$) up to constant
factor the restriction of $y$ to $C$
behaves as $x^{-k/(k+l)}$ (resp. $x^{l/(k+l)}$). This implies
that the projection of $C$ to the $x$-axis has ramification points of
order $k+l$ over $x=0$ and $x= \infty$.
Differentiating (8.1) with respect to $y$ we get
$$-k(1-x-y)^{k-1}(1-x)^l=c(k+l) x^ky^{k+l-1}.$$ Dividing (8.1) by this equality
we get $1-x = \alpha y$ where $\alpha =l/(k+l)$. Plugging $\alpha y $ instead
of $1-x$ into (8.1) we
obtain a polynomial in $y$ of degree $2k+l$ which yields additional
ramification points. Since the projection of $C$ to the
$x$-axis is $(k+l)$-sheeted the Riemann-Hurwitz formula implies that the genus
of $C$ is positive which contradicts to the fact that $C$ is isomorphic to $\C^*$.

Now assume that $L_4$ and $L_5$ are contained in different fibers of $f$.
Using again Lemma \ref{8.70} below we can see that at least one of these two
fibers is irreducible. That is, up to a linear transformation $f$ must coincide
with either $z$ or $xz+1$.

By considering the pair $(L_2, L_5)$ instead of $(L_4,L_5)$ we get also
$f=yz+1$ as a possibility which concludes the proof.

\end{proof}

To make the proof of Proposition \ref{8.60}  complete we need the following.

\begin{lemma}\label{8.70}  Let $f: Y \to B$ be a $\C^*$-fibration (i.e. the general fibers
of $f$ are isomorphic to $\C^*$) of a factorial affine algebraic surface over the base $B$ isomorphic to a line with
$k$ deleted points. Then the dimension of $H_1 (Y, \Q)$ is at most $k+1$.

\end{lemma}

\begin{proof} Since $Y$ is factorial the removal of an irreducible curve
from $Y$ increases the number of generator of the group of regular invertible
functions by 1. Therefore, this procedure increases the dimension of the first rational
homology by 1 (e.g., see \cite{Fuj}). Thus it suffices to prove the statement of Lemma in the case
of $f$ being locally trivial since we can remove the singular fibers. If $Y$ is the direct product $B\times \C^*$
then $\dim H_1(Y, \Q) =k+1$ by the Kunneth formula.
In the general case the computation goes via the Leray spectral sequence
but the dimension of homology can only decrease compared with the case of the direct product. This concludes the proof.

\end{proof}

\begin{theorem}\label{8.80}
 Let $\nu$ be a completely integrable algebraic vector field on $S$.
Then either

{\rm (1)}   $\nu$ coincides with a vector field of form
 $$q(z) ((1+xz){\frac{\partial}{\partial x}} - (1+yz) {\frac{\partial}{\partial y}} )$$ where $q$
is a polynomial or

{\rm (2)} the restriction of $\nu$ to $T=S \setminus (L_1 \cup L_3) \subset \C^2_{x,y}$ is a vector
fields of form $$p(x^ky^l) ( lx {\frac{\partial}{\partial x}} -ky {\frac{\partial}{\partial y}} )$$
where $p$ is a polynomial with $p(0)=0$ and $k$ and $l$ are relatively prime nonnegative
integers.

\end{theorem}

\begin{proof} By Proposition \ref{8.40} we have a regular function $P$ on $S$ with general
fibers isomorphic to $\C^*$ such that $P$ is invariant with respect to the flow of $\nu$.
All such possible $P$'s are listed in Proposition \ref{8.60}. Suppose that $P=z$.
Then this regular function has two singular fibers. Hence each individual
fiber of $P$ is preserved by the flow, i.e. $P$ is the first integral. It remains to
note that every field that has general fibers of $P$ as orbits is of form (1).

Now let $P=x^ky^l$ with natural relatively prime $k$ and $l$, $$\nu_1=x\partial / \partial x,
\, \, \, {\rm and} \, \, \, \nu_2= lx \partial / \partial x
-k y \partial / \partial y.$$
Repeating the argument
from \cite[Proposition 2 (2)]{Bru04} we trivialize a neighborhood of a regular fiber
of $P$, isomorphic to $\Delta \times \C_\zeta^*$, in such a way that $\nu_1$ is sent to $ \partial / \partial \xi$
(where $\xi$ is a coordinate on the unit disc $\Delta$) and $\nu_2$ is sent to $\zeta \partial / \partial \zeta$.
Extending the flow of $\nu$
to $\Delta \times \C \proj^1$ we find that $\nu$ is of form $\beta (\xi ) \partial /\partial \xi
+ \alpha (\xi )\zeta \partial / \partial \zeta$ %where $\xi = x^ky^l$ is a coordinate on $\Delta$
(because
this extension must be tangent to the curves $\zeta =0$ and $\zeta= \infty$). Hence $\nu$ is of the
form $$\nu = q(x^ky^l)\nu_1 + p(x^ky^l) \nu_2$$ where both functions $p$ and $q$ are regular
on $\C^*_\xi$. However, if we want this rational field on $\C^2$ to be lifted
regularly on $S$ we need to require that it vanishes at points $(1,0)$ and $(0,1)\in\C^2$.
This implies that $p$ and consequently $q$ are polynomials on $\C_\xi$ and $p(0)=q(0)=0$.
Furthermore, since $\nu$ is completely integrable $q$ must be constant and we have the desired form.

In the case of $P$ equal to $xz+1$ or $yz+1$ the same argument produces a field as in (2)
but with the pair $(k,l)$ equal to $(1,0)$ or $(0,1)$ which concludes the proof.

\end{proof}

\begin{corollary}\label{8.90} Every completely integrable algebraic vector field on $S$ has
divergence zero with respect to the algebraic volume form $\omega ={\frac{{\rm d} x}{x}} \wedge {\frac{{\rm d} y}{y}}$.
In particular, $S$ has no algebraic density property.

\end{corollary}

\begin{proof} First note that the restriction of vector fields of type (2) from Theorem \ref{8.80} to the
torus $T = S \setminus (L_1 \cup L_3)$ is completely integrable. Thus their divergence with
respect to $\omega$ is zero \cite{A00}. It can be checked also by the direct
computation which we do for the fields on type (1). We need the following two formulas (e.g., see \cite{KN1})
$$ \diver_{\omega} (f\nu ) = f \diver_{\omega} (\nu ) + \nu (f) \, \, \, {\rm and} \, \, \,
 \diver_{f\omega} (\nu ) =  \diver_{\omega} (\nu ) + L_\nu (f)/f $$ where $\nu$ is
 a vector field, $f$ is a regular function, and $L_\nu$ is the Lie derivative.
 In particular, $\diver_{\omega} (f\nu ) = 0$ provided $\diver_\omega (\nu)=0$ and $\nu (f)=0$.
 Thus for the fields of type (1) it suffices to consider  $\nu= (1+xz){\frac{\partial}{\partial x}} - (1+yz) {\frac{\partial}{\partial y}} $
 because $z$ belongs to the kernel of $\nu$. Since $1+xz=(1-x)/y$ and $1+yz = (1-y)/x$ the restriction
 of $\nu$ to $T$ can be written in the from
 $$ {\frac{1-x}{y}}{\frac{\partial}{\partial x}} - {\frac{1-y}{x}} {\frac{\partial}{\partial y}}. $$
The divergence of the first summand $ {\frac{1-x}{y}}{\frac{\partial}{\partial x}}$
with respect to the standard form $\omega_0= {\rm d} x \wedge {\rm d} y$
is $-1/y$. Since $\omega ={\frac{1} {xy}}\omega_0$ its divergence with respect to $\omega$ is
$$-1/y +L_\nu ({\frac{1} {xy}} )/ ({\frac{1} {xy}})=-1/y -(1-x)/(xy)=-1/(xy).$$
Similarly, the divergence of the second summand $ {\frac{1-y}{x}}{\frac{\partial}{\partial y}}$ is
$-1/(xy)$ and $\diver_\omega (\nu )=0$ which implies the first statement.

Since $\diver_\omega ([\nu_1, \nu_2])= L_{\nu_1} (\diver_\omega (\nu_2)) -  L_{\nu_2} (\diver_\omega (\nu_1))$
(e.g., see \cite{KN1}) the Lie algebra generated by completely integrable algebraic vector fields contains
only fields of divergence zero with respect to $\omega$. Thus algebraic vector fields with
nonzero divergence cannot be contained in this algebra which implies the absence of the
algebraic density property.

\end{proof}

It remains to show that the group of holomorphic automorphisms of $S$ generated by completely
integrable algebraic vector fields is $m$-transitive for any natural $m$. For this consider
the fields $$(1+xz){\frac{\partial}{\partial x}} - (1+yz) {\frac{\partial}{\partial y}}, \, \,
xy{\frac{\partial}{\partial x}} - (1+yz) {\frac{\partial}{\partial z}}, \, \, \, {\rm and} \, \, \,$$
$$ xy{\frac{\partial}{\partial x}} - (1+xz) {\frac{\partial}{\partial z}}$$ on $S$. They are completely
integrable fields from Theorem \ref{8.80} (the first one is of type (1) and the last two are of type
(2) with the pair $(k,l)$ equal to $(0,1)$ and $(1,0)$ respectively), they generate the tangent space of $S$
at each point, and their kernels in $\C [S]$ are $\C [z], \C [y]$, and $\C [x]$ respectively.
Therefore, the desired
conclusion follows from the next fact.

\begin{proposition}\label{8.100} Let $X$ be a smooth affine algebraic variety and $\sigma_1, \ldots , \sigma_k$
be completely integrable nonzero algebraic vector fields on $X$. Suppose that the kernel of $\sigma_j$ in $\C [X]$ is
denoted by  $\Ker \sigma_j$ and the following conditions hold

{\rm (i)} $\sigma_1, \ldots , \sigma_k$ generate tangent space at any point of $X$;

{\rm (ii)}
there exist regular functions $f_1 \in \Ker \sigma_1, \ldots , f_k \in \Ker \sigma_k$
such that each $f_j$ is not contained in $\bigcup_{l\ne j} \Ker \sigma_l$ and
the map $(f_1, \ldots , f_k) : X \to \C^k$ is an embedding.
%the map $\varphi =(f_1, \ldots , f_n) : X \to \C^n$ is birational and such that

%for any point $x\in X$ where $\varphi$ is not an embedding and
%for every $j \leq n$ there exists $n+1 \leq l \leq k$ such that the field $\sigma_l$ is not
%tangent to the fiber of $f_j$ through $x$.

Then the group $G$ of holomorphic automorphisms generated by elements of the flows of completely
integrable algebraic vector fields
is $m$-transitive for every natural $m$.

\end{proposition}

\begin{proof}
Let us start with the following.

{\em Claim.} Let $X_*^m$ be the subset of points $(x_1,x_2, \ldots , x_m)$  in $X^m$ such that
for any $j=1, \ldots , k$ we have $$f_j (x_i) \ne f_j (x_l) \eqno{(8.2})$$ where $i \ne l$.
Then its $G$-orbit $O \subset X^m$ contains
all points of form $(x,x_2, \ldots , x_m)$ where $x$ runs over
some neighborhood $U$ of $x_1$ in $X$.

Indeed, let $p_j \, \, (j=1, \ldots , n)$ be a polynomial in one variable such that
the function $p_j(f_j)$ vanishes
at all points $x_2, \ldots , x_m$ but not at $x_1$. Then the elements of flows of completely
integrable vector fields $p_1(f_1)\sigma_1, \ldots , p_k(f_k)\sigma_k$ keeps points $x_2, \ldots , x_n$ fixed.
On the other hand since these fields generate the tangent space $T_{x_1}X$ by (i),
compositions of their exponents can send $x_1$
into any given point in a small neighborhood $U$ of $x_1$ which yields the Claim.

Note that $U$ can chosen so that it depends continuously on $(x_1,x_2, \ldots , x_n)$ in a
neighborhood of this point in $X_*^m$, and we can also repeat
this Claim with any $x_j$ instead of $x_1$.
This implies immediately that $O\cap X_*^m$ is an open subset of $X_*^m$. Furthermore, it is closed in $X_*^m$ because
if $(x_1',x_2', \ldots , x_m')$ belongs to its closure then the open set
$O' \cap X_*^m$ meets $O$ where $O'$ is the $G$-orbit of $(x_1',x_2', \ldots , x_m')$.
Thus $O$ contains $X_*^m$.

Now let $(x_1, x_2, \ldots , x_m)\in X^m$ be an $m$-tuple of distinct points that does not satisfy (8.2).
Say $f_j(x_1)=f_j(x_2)$ for $j \leq n$ and $f_i (x_1) \ne f_i(x_2)$ for $n+1\leq i \leq k$ (note that
$n<k$ because of assumption (ii)). It suffices to show that a perturbation of $(x_1, x_2, \ldots , x_m)$ by an element
of $G$ close to the identical automorphism destroys at least one of these equalities. Though a priori
a perturbation may not destroy the equalities it enables us to assume that $\sigma_k$ is
not tangent at $x_1$ to the fiber of $f_1$ (since otherwise $f_1 \in \Ker \sigma_k$ contrary to (ii)).
Choose again a polynomial $p_k$ such that $p_k(f_k)$ vanishes at $x_2$ but not at $x_1$.
Then the flow of $p_k(f_k)\sigma_k$ changes the value
of $f_1(x_1)$ while keeping the point $x_2$ fixed which concludes the proof.

\end{proof}

\providecommand{\bysame}{\leavevmode\hboxto3em{\hrulefill}\thinspace}

\end{document}